\documentclass[11pt]{article}
\usepackage[margin=1in]{geometry} 
\usepackage{amssymb,amsfonts,amsmath,bbm,mathrsfs,stmaryrd,mathtools,mathabx}
\usepackage{xcolor}
\usepackage{url}

\usepackage{enumerate}

\usepackage{hyperref}
\hypersetup{colorlinks,
             linkcolor=black!75!red,
             citecolor=blue,
             pdftitle={},
             pdfproducer={pdfLaTeX},
             pdfpagemode=None,
             bookmarksopen=true
             bookmarksnumbered=true}

\usepackage{tikz}
\usetikzlibrary{arrows,calc,decorations.pathreplacing,decorations.markings,intersections,shapes.geometric,through,fit,shapes.symbols,positioning,decorations.pathmorphing}

\usepackage{braket}

\usepackage[amsmath,thmmarks,hyperref]{ntheorem}
\usepackage{cleveref}

\creflabelformat{enumi}{#2(#1)#3}

\crefname{section}{Section}{Sections}
\crefformat{section}{#2Section~#1#3} 
\Crefformat{section}{#2Section~#1#3} 

\crefname{subsection}{\S}{\S\S}
\AtBeginDocument{%
  \crefformat{subsection}{#2\S#1#3}%
  \Crefformat{subsection}{#2\S#1#3}%
}

\crefname{subsubsection}{\S}{\S\S}
\AtBeginDocument{%
  \crefformat{subsubsection}{#2\S#1#3}%
  \Crefformat{subsubsection}{#2\S#1#3}%
}

\theoremstyle{plain}

\newtheorem{lemma}{Lemma}[section]
\newtheorem{proposition}[lemma]{Proposition}
\newtheorem{corollary}[lemma]{Corollary}
\newtheorem{theorem}[lemma]{Theorem}

\theoremstyle{nonumberplain}
\newtheorem{theoremN}{Theorem}

\theoremstyle{plain}
\theorembodyfont{\upshape}
\theoremsymbol{\ensuremath{\blacklozenge}}

\newtheorem{definition}[lemma]{Definition}

\newtheorem{remark}[lemma]{Remark}
\newtheorem{remarks}[lemma]{Remarks}

\newtheorem{notation}[lemma]{Notation}
\newtheorem{construction}[lemma]{Construction}

\crefname{definition}{definition}{definitions}
\crefformat{definition}{#2definition~#1#3} 
\Crefformat{definition}{#2Definition~#1#3} 

\crefname{ex}{example}{examples}
\crefformat{example}{#2example~#1#3} 
\Crefformat{example}{#2Example~#1#3} 

\crefname{exs}{example}{examples}
\crefformat{examples}{#2example~#1#3} 
\Crefformat{examples}{#2Example~#1#3} 

\crefname{remark}{remark}{remarks}
\crefformat{remark}{#2remark~#1#3} 
\Crefformat{remark}{#2Remark~#1#3} 

\crefname{remarks}{remark}{remarks}
\crefformat{remarks}{#2remark~#1#3} 
\Crefformat{remarks}{#2Remark~#1#3} 

\crefname{convention}{convention}{conventions}
\crefformat{convention}{#2convention~#1#3} 
\Crefformat{convention}{#2Convention~#1#3} 

\crefname{notation}{notation}{notations}
\crefformat{notation}{#2notation~#1#3} 
\Crefformat{notation}{#2Notation~#1#3} 

\crefname{table}{table}{tables}
\crefformat{table}{#2table~#1#3} 
\Crefformat{table}{#2Table~#1#3}

\crefname{lemma}{lemma}{lemmas}
\crefformat{lemma}{#2lemma~#1#3} 
\Crefformat{lemma}{#2Lemma~#1#3} 

\crefname{proposition}{proposition}{propositions}
\crefformat{proposition}{#2proposition~#1#3} 
\Crefformat{proposition}{#2Proposition~#1#3} 

\crefname{corollary}{corollary}{corollaries}
\crefformat{corollary}{#2corollary~#1#3} 
\Crefformat{corollary}{#2Corollary~#1#3} 

\crefname{theorem}{theorem}{theorems}
\crefformat{theorem}{#2theorem~#1#3} 
\Crefformat{theorem}{#2Theorem~#1#3} 

\crefname{enumi}{}{}
\crefformat{enumi}{(#2#1#3)}
\Crefformat{enumi}{(#2#1#3)}

\crefname{assumption}{assumption}{Assumptions}
\crefformat{assumption}{#2assumption~#1#3} 
\Crefformat{assumption}{#2Assumption~#1#3} 

\crefname{construction}{construction}{Constructions}
\crefformat{construction}{#2construction~#1#3} 
\Crefformat{construction}{#2Construction~#1#3} 

\crefname{equation}{}{}
\crefformat{equation}{(#2#1#3)} 
\Crefformat{equation}{(#2#1#3)}

\numberwithin{equation}{section}

\theoremstyle{nonumberplain}
\theoremsymbol{\ensuremath{\blacksquare}}

\newtheorem{proof}{Proof}

\newcommand\bG{{\mathbb G}}
\newcommand\bH{{\mathbb H}}

\newcommand\bP{{\mathbb P}}

\newcommand\cC{{\mathcal C}}

\newcommand\cM{{\mathcal M}}

\newcommand\cO{{\mathcal O}}

\DeclareMathOperator{\id}{id}

\DeclareMathOperator{\Hom}{\mathrm{Hom}}

\DeclareMathOperator{\Alg}{\cat{Alg}}
\DeclareMathOperator{\Coalg}{\cat{Coalg}}
\DeclareMathOperator{\Bialg}{\cat{Bialg}}
\DeclareMathOperator{\HAlg}{\cat{HAlg}}
\DeclareMathOperator{\vc}{\cat{Vec}}
\DeclareMathOperator{\cqg}{\cat{CQG}}

\newcommand{\cat}[1]{\textsc{#1}}

\newcommand{\qedhere}{\mbox{}\hfill\ensuremath{\blacksquare}}

\renewcommand{\square}{\mathrel{\Box}}

\newcommand{\xrightarrowdbl}[2][]{%
  \xrightarrow[#1]{#2}\mathrel{\mkern-14mu}\rightarrow
}

\title{Epimorphic quantum subgroups and coalgebra codominions}
\author{Alexandru Chirvasitu}

\begin{document}

\date{}

\newcommand{\Addresses}{{
  \bigskip
  \footnotesize

  \textsc{Department of Mathematics, University at Buffalo}
  \par\nopagebreak
  \textsc{Buffalo, NY 14260-2900, USA}  
  \par\nopagebreak
  \textit{E-mail address}: \texttt{achirvas@buffalo.edu}


}}

\maketitle

\begin{abstract}
  We prove a number of results concerning monomorphisms, epimorphisms, dominions and codominions in categories of coalgebras. Examples include: (a) representation-theoretic characterizations of monomorphisms in all of these categories that when the Hopf algebras in question are commutative specialize back to the familiar necessary and sufficient conditions (due to Bien-Borel) that a linear algebraic subgroup be epimorphically embedded; (b) the fact that a morphism in the category of (cocommutative) coalgebras, (cocommutative) bialgebras, and a host of categories of Hopf algebras has the same codominion in any of these categories which contain it; (c) the invariance of the Hopf algebra or bialgebra (co)dominion construction under field extension, again mimicking the well-known corresponding algebraic-group result; (d) the fact that surjections of coalgebras, bialgebras or Hopf algebras are regular epimorphisms (i.e. coequalizers) provided the codomain is cosemisimple; (e) in particular, the fact that embeddings of compact quantum groups are equalizers in the category thereof, generalizing analogous results on (plain) compact groups; (f) coalgebra-limit preservation results for scalar-extension functors (e.g. extending scalars along a field extension $\Bbbk\le \Bbbk'$ is a right adjoint on the category of $\Bbbk$-coalgebras). 
\end{abstract}

\noindent {\em Key words: Hopf algebra; coalgebra; comodule; epimorphism; monomorphism; dominion; codominion; pullback; locally presentable; adjoint functor; algebraic group; cocommutative; CQG algebra}

\vspace{.5cm}

\noindent{MSC 2020: 18A20; 16T05; 16T15; 20G42; 18A30; 20G05; 20G15; 18C15; 18C20; 16D40}

\tableofcontents

\section*{Introduction}

A possible starting point for the considerations below is the familiar problem of determining, for a {\it concrete category} (i.e. one consisting of sets and functions \cite[Definition 5.1]{ahs}), whether the epimorphisms are precisely the surjections. Recall \cite[Definition 7.3.9]{ahs} that an {\it epimorphism} in a category $\cC$ is a morphism $f:c\to c'$ for which every precomposition map
\begin{equation*}
  \cC(c',c'')\xrightarrow{-\circ f}\cC(c,c'')
\end{equation*}
is one-to-one. This is one natural category-theoretic generalization of the notion of surjectivity, and for concrete categories, where both concepts make sense, the latter plainly implies the former.

The problem recurs in endless variations throughout a massive amount of literature this introduction cannot possibly do justice. One example would be \cite{reid-epi}, which addresses the issue for several categories (e.g. compact groups and Lie algebras, most relevantly for the discussion below, but also others, such as $C^*$- or von Neumann algebras). Other examples include, say, the earlier \cite{pogunt_epi-cpct} (also compact groups), \cite{hn-epi} ($C^*$-algebras), \cite{bg_epi-lie} (Lie algebras, $p$-Lie algebras), \cite{bb1,bb2} (linear algebraic groups), \cite{brion_epi} (algebraic groups in general), or the series \cite{isb_epidom-1,isb_epidom-2,isb_epidom-3,isb_epidom-4,isb_epidom-5}, which systematically studies the broader notion of {\it dominion} \cite[14J]{ahs} in various categories (semigroups, associative algebras): assuming the category being studied is sufficiently well-behaved, the dominion of a morphism $f:c\to c'$ is the equalizer of the two structure maps
\begin{equation*}
  \begin{tikzpicture}[auto,baseline=(current  bounding  box.center)]
    \path[anchor=base] 
    (0,0) node (l) {$c'$}
    +(3,0) node (r) {$c'\coprod_cc'$}
    ;
    \draw[->] (l) to[bend left=6] node[pos=.5,auto] {$\scriptstyle $} (r);
    \draw[->] (l) to[bend right=6] node[pos=.5,auto] {$\scriptstyle $} (r);
  \end{tikzpicture}
\end{equation*}
into the self-pushout of $f$. The connection between the two notions is that $f:c\to c'$ is epic precisely if its dominion is $\id:c'\to c'$. Dominions make an oblique appearance in \cite{bb1} too, under a different name: for a closed linear algebraic subgroup $\iota:H\le G$, the {\it observable envelope} \cite[\S 3]{bb1} of $H$ in $G$ is precisely the dominion of $\iota$ in the category of linear algebraic groups (see the discussion following \Cref{re:charpok}).

One of the main results is that the much of the familiar characterization \cite[Th\'eor\`eme 1]{bb1} of epimorphisms in the category of linear algebraic groups goes through for Hopf algebras, taking into account the fact that one must dualize (the category of linear algebraic groups over a field $\Bbbk$ being {\it contravariantly} equivalent to that of Hopf $\Bbbk$-algebras \cite[\S 1.4, Theorem]{water_bk}). This is the content of \Cref{th:monohopf} below:

\begin{theoremN}
  For a morphism $\pi:H\to H'$ of coalgebras over a field $\Bbbk$ the following conditions are equivalent:
  \begin{enumerate}[(a)]
  \item $\pi$ is monic;
  \item The corestriction functor induced by $\pi$ between categories of (finite-dimensional) comodules is full.
  \item For any (finite-dimensional) $H$-comodule $V$ every idempotent morphism on $\cM^{H'}$ is in $\cM^H$.
  \item For any (finite-dimensional) $H$-comodule $V$ every direct-sum decomposition $V\cong V_1\oplus V_2$ as $H'$-comodules is also one of $H$-comodules.

    If furthermore $\pi$ is a morphism of Hopf algebras, these are also equivalent to
  \item For any (finite-dimensional) right $H$-comodule $V$ with comodule structure map
    \begin{equation*}
      V\ni v\xmapsto{} v_0\otimes v_1\in V\otimes H
    \end{equation*}
    the space
    \begin{equation*}
      V^{H'}:=\{v\in V\ |\ v_0\otimes \pi(v_1) = v\otimes 1\}
    \end{equation*}
    of {\it $H'$-coinvariants} coincides with the analogous space of $H$-coinvariants.
  \item We have $H^{H'}=\Bbbk$ (i.e. the only $H'$-coinvariants in $H$ are the scalars).  \qedhere
  \end{enumerate}
\end{theoremN}

The statement above is somewhat vague on which category (coalgebras or Hopf algebras) the morphisms are supposed to be monic in. The same issue arises in trying to relate the statement back to algebraic groups: the latter form a category dual to that of {\it commutative} Hopf algebras, so it is reasonable to wonder whether a morphism contained in two of the categories of interest might be monic in one but not the other. \Cref{pr:commnodiff} shows that this is not an issue, and \Cref{th:same-co-dom} expands this to (co)dominions (given a morphism in any of the relevant categories, it does not matter where one computes its codominion):

\begin{theoremN}
  \begin{enumerate}[(1)]
  \item\label{item:intro-samedom} The dominion of an algebra morphism is the same in any of the following categories that happen to contain it:
    \begin{itemize}
    \item commutative algebras;
    \item (commutative) bialgebras;
    \item (commutative or cocommutative) Hopf algebras;
    \item Hopf algebras with involutive or bijective antipode.
    \end{itemize}
    In particular, a morphism is or is not epic simultaneously in all of these categories containing it.
  \item\label{item:intro-samecodom} The codominion of an algebra morphism is the same in any of the following categories that happen to contain it:
    \begin{itemize}
    \item cocommutative coalgebras;
    \item (commutative) bialgebras;
    \item (commutative or cocommutative) Hopf algebras;
    \item Hopf algebras with involutive or bijective antipode.
    \end{itemize}
    In particular, a morphism is or is not monic simultaneously in all of these categories containing it.  \qedhere
  \end{enumerate} 
\end{theoremN}

Another theme cropping up in the literature is the independence of the various notions discussed above (dominions, epimorphisms) of the base field: \cite[Proposition 5.3]{bg_epi-lie} is a case in point (dominions of finite-dimensional Lie-algebra morphisms are preserved by field extensions), as is \cite[Theorem 5 (ii)]{brion_epi} (where the category is that of algebraic groups instead). In this same spirit, field extensions preserve not only codominions (\Cref{pr:fieldextcodom}), but also arbitrary limits in categories of objects equipped with `coalgebraic structure` (\Cref{th:fieldextfinlim} \Cref{item:coalg}):

\begin{theoremN}
  Let $\Bbbk\le \Bbbk'$ be a field extension.
  \begin{enumerate}[(1)]
  \item The functor $\Bbbk'\otimes_{\Bbbk}-$ between any of the following categories (of objects linear over $\Bbbk$ and $\Bbbk'$ respectively) is a right adjoint and in particular continuous:
    \begin{itemize}
    \item (cocommutative) coalgebras;
    \item (commutative or cocommutative) bialgebras;
    \item (commutative or cocommutative Hopf algebras);
    \item Hopf algebras with involutive or bijective antipode.
    \end{itemize}
  \item The same scalar-extension functor $\Bbbk'\otimes_{\Bbbk}-$ preserves (co)dominions in all of the categories above, along with those of commutative or plain algebras.  \qedhere
  \end{enumerate}
\end{theoremN}

Occasionally, an object will have the property that any monomorphism having it as domain is automatically a dominion (these are, for instance, the {\it absolutely closed} algebras of \cite[discussion preceding Corollary 1.7]{isb_epidom-1}). Examples suggest a general phenomenon whereby ``semisimplicity entails absolute closure'': \cite[Addenda]{bg_epi-lie} reports a remark of Hochschild's to the effect that embeddings of semisimple Lie algebras are dominions in the category of finite-dimensional Lie algebras, and the analogue for linearly reductive algebraic groups follows from standard GIT theory (e.g. \cite[\S 1.2]{fkm}).

Dualizing the picture once more, morphisms onto {\it cosemisimple} coalgebras are codominions. Several versions of this are recorded in \Cref{cor:coflatcodom,cor:ontocoss} and \Cref{th:onesplit}. Focusing, for simplicity, on Hopf algebras (rather than bialgebras, etc.), a sampling of those results reads as follows:

\begin{theoremN}
  \begin{enumerate}[(1)]
  \item A morphism $\pi:H\to H'$ of Hopf algebras is a codominion if either
    \begin{itemize}
    \item $H$ is left or right faithfully coflat over $H'$;
    \item or $\pi$ has a right inverse in the category of either right or left $H'$-comodules.
    \end{itemize}
  \item In particular, surjections onto cosemisimple Hopf algebras are codominions in the category of Hopf algebras.  \qedhere
  \end{enumerate}
\end{theoremN}

These considerations also deliver a non-commutative version of the result that embeddings of compact groups are equalizers \cite[Theorem 2.1]{chi_multipush}. Recall (\cite[Definition 2.2]{dk_cqg}, \Cref{def:cqg-alg} below) that CQG algebras are non-commutative analogues of the (Hopf) algebras of {\it representative functions} \cite[\S III.1, Definition 1.1]{bd_lie} on compact groups. With this in mind, and paraphrased in dual language, \Cref{th:cqgcodom} says that embeddings of ``compact quantum groups'' are equalizers:

\begin{theoremN}
  Surjections in the category of CQG algebras are coequalizers.  \qedhere
\end{theoremN}

\subsection*{Acknowledgements}

This work was partially supported through NSF grant DMS-2001128. 

I am grateful for assorted comments and pointers to the literature from A. Agore, M. Brion, T. Brzezinski, G. Militaru and R. Wisbauer. 

\section{Preliminaries}\label{se:prel}

Algebras (always associative and unital), coalgebras (coassociative and counital) and the like will typically be assumed linear over a field. Relevant background on coalgebras, bialgebras and Hopf algebras is available in a number of excellent sources such as \cite{swe,montg_hopf,dnr,rad}, and the text contains more specific references to those where needed.

Some prominent notation:
\begin{itemize}

\item As usual (e.g. \cite[Definition 1.5.1]{montg_hopf}), $\Delta$, $\varepsilon$ and $S$ typically denote the comultiplication, counit and respectively antipode of a Hopf algebra (which one will typically be clear from context). 

\item We use {\it Sweedler notation} \cite[Notation 1.4.2]{montg_hopf} for comultiplications and (left or right) comodule structures:
  \begin{equation*}
    \begin{aligned}
      C\ni c &\xmapsto{\quad} c_1\otimes c_2\in C\otimes C\\
      V\ni v &\xmapsto{\quad} v_0\otimes v_1\in V\otimes C\\
      V\ni v &\xmapsto{\quad} v_{-1}\otimes v_0\in C\otimes V.
    \end{aligned}
  \end{equation*}
  
\item $\Alg$, $\Coalg$, $\Bialg$ and $\HAlg$ denote the categories of algebras, coalgebras, bialgebras and Hopf algebras respectively, over a field $\Bbbk$ fixed beforehand and usually implicit.

  We might occasionally adorn the symbols with the field as a subscript, when wishing to emphasize it (e.g. $\Alg_{\Bbbk}$).

\item Various other subscripts modify those categories to indicate additional properties:
  \begin{itemize}
  \item `c' stands for `commutative' (e.g. $\Alg_c$);
  \item `cc' for `cocommutative' \cite[Definition 1.1.3]{montg_hopf} (e.g. $\Bialg_{\Bbbk,cc}$);
  \item while for Hopf algebras, a left-hand subscript might indicate antipode bijectivity or involutivity: ${}_{bi}\HAlg$ is the category of Hopf algebras with bijective antipode, while ${}_{2}\HAlg$ that of Hopf algebras whose antipode squares to the identity.
  \item We will occasionally have to refer to these subscripts in bulk, in which case we use non-alphanumeric symbols such as `$\square$' or `$\bullet$'. $\Alg_{\bullet}$, for instance, indicates both $\Alg$ and $\Alg_c$ in one go, while ${}_{\square}\HAlg_{\bullet}$ refers collectively to ${}_{bi}\HAlg$, $\HAlg_{c}$, $\HAlg_{c,cc}$, etc.

  \item As a variation of this, we might need to limit the possibilities for what one of the subscripts might be, in which case we separate the various options by a `/' symbol: $\HAlg_{c/}$ means the category of either commutative or arbitrary Hopf algebras (but not that of {\it co}commutative Hopf algebras). 
  \end{itemize}
  
\item Categories of modules are depicted as `$\cM$' adorned with subscripts indicating the base algebra, with the handedness of the placement indicating whether they are right or left modules: $\cM_A$, say, means right $A$-modules. This applies also to two-sided structures: ${}^C\cM^{D}$, for instance, is the category of $(C,D)$-bicomodules for coalgebras $C$ and $D$.
  
  One exception will be $\cat{Vec}=\cat{Vec}_{\Bbbk}$ (rather than $\cM_{\Bbbk}$), the category of vector spaces over $\Bbbk$.

\item Similarly, {\it co}module \cite[Definition 1.6.2]{montg_hopf} categories (over coalgebras, bialgebras, etc.) exhibit the base coalgebra as a {\it super}script, again left or right: $\cM^C$ is the category of $C$-comodules.

\item Additional `$f$' subscripts indicate finite-dimensionality: $\cM^H_f$ would be the category of finite-dimensional $H$-comodules.
\end{itemize}

Some category-theoretic terminology will feature quite extensively; we refer the reader to, say, \cite{bw,ahs,mcl}, again with more specific citations as needed. In general, for objects $c,c'\in \cC$ of a category $\cC$, we write $\cC(c,c')$ for their space of morphisms.

One remark that is worth making now, as it pervades more or less the entire discussion, is that apart from those imposing size conditions (i.e. the categories of {\it finite-dimensional} (co)modules), every one of the categories mentioned above ($\Alg$, $\Coalg$, $\HAlg$, their (co)commutative versions, etc. etc.) is extremely well-behaved. Specifically, they are all {\it locally presentable} in the sense of \cite[Definition 1.17]{ar}:

\begin{itemize}
\item Module and comodule categories are {\it Grothendieck} \cite[\S 2.8, Note 3 following Corollary 8.13]{pop} (see \cite[Corollary 2.2.8]{dnr} for categories of comodules), and hence locally presentable \cite[Corollary 5.2]{krau_ausl}.
 
\item The other categories of ``algebraic structures'' (i.e. sets equipped with operations of various arities, satisfying various equations; algebras and commutative algebras, say) this is well known and a uniform consequence of, say \cite[Theorem 3.28]{ar}.

\item For the categories of ``coalgebraic'' or ``mixed'' structures (coalgebras, bialgebras, Hopf algebras) this has been worked out in ample detail in a number of papers of Porst's; see for instance \cite[Lemmas 1 and 2, Theorem 6 and Proposition 22]{porst_formal-2}. 
\end{itemize}

In particular, by \cite[Remark 1.56]{ar} all of these categories
\begin{itemize}
\item are {\it (co)complete} (they have arbitrary (co)limits \cite[Definition 12.2]{ahs});
\item are {\it well-powered} (the isomorphism classes of {\it subobjects} \cite[Definition 7.77]{ahs} of any given object form a set \cite[Definition 7.82]{ahs});
\item and also {\it co-well-powered} (the isomorphism classes of {\it quotient objects} \cite[Definition 7.84]{ahs} of any given object form a set \cite[Definition 7.87]{ahs});
\item and hence satisfy a host of convenient {\it adjoint functor theorems} \cite[\S 18]{ahs}:
  \begin{itemize}
  \item {\it cocontinuous} (i.e. colimit-preserving) functors defined on any of these categories are left adjoints \cite[\S 5.8, Corollary on p.130]{mcl}.
  \item and {\it continuous} (i.e. limit-preserving) functors having them as domains are right adjoints provided they also preserve {\it $\kappa$-directed colimits} \cite[Definition 1.13 (1)]{ar} for some regular cardinal $\kappa$ \cite[Theorem 1.66]{ar}.
  \end{itemize}
\end{itemize}

\section{Epimorphic quantum subgroups and (co)dominions}\label{se:codom}

As we will be concerned with monomorphisms and epimorphisms of bialgebras, Hopf algebras, etc., it will be worth noting that the notions are invariant under a number of convenient modifications. These are very simple remarks, but it will be helpful to have made them explicitly. 

\begin{proposition}\label{pr:commnodiff}

  \begin{enumerate}[(1)]
  \item A morphism $f:H\to H'$ in $\Alg_{\Bbbk}$ is epic in that category if and only if it is epic in any of the following categories which contain it:
    \begin{equation*}
      \Alg_{\bullet},\quad \Bialg_{\bullet},\quad {}_{\square}\HAlg_{\bullet}.
    \end{equation*}

  \item\label{item:5} A morphism $f:H\to H'$ in $\Coalg_{\Bbbk}$ is monic in that category if and only if it is monic in any of the following categories which contain it:
    \begin{equation*}
      \Coalg_{\bullet},\quad \Bialg_{\bullet},\quad {}_{\square}\HAlg_{\bullet}.
    \end{equation*}
  \end{enumerate}

\end{proposition}
\begin{proof}

  
  We focus on \Cref{item:5}, to fix ideas; the other half is entirely analogous (and categorically dual).
  
  The functors
  \begin{equation*}
    \HAlg_c
    \xrightarrow[]{\text{inclusion}}
    {}_2\HAlg
    \xrightarrow[]{\text{inclusion}}
    {}_{bi}\HAlg
    \xrightarrow[]{\text{inclusion}}
    \HAlg
    \xrightarrow[]{\text{inclusion}}
    \Bialg
    \xrightarrow[]{\text{forget}}
    \Coalg
  \end{equation*}
  are right adjoints: see \cite[diagram (9)]{porst_formal-2} for the inclusions and \cite[\S 2.2]{porst_univ} for the last arrow. They thus preserve monomorphisms \cite[Proposition 18.6]{ahs}, and also reflect them because they are faithful.
  

  Exactly the same reasoning applies to
  \begin{itemize}
  \item the inclusion functors from commutative bialgebras (Hopf algebras) to plain bialgebras (respectively Hopf algebras), which are also faithful right adjoints by \cite[diagram (9)]{porst_formal-2};
  \item and the chain of functors
    \begin{equation*}
      \HAlg_{c,cc}
      \xrightarrow[]{\text{inclusion}}
      \HAlg_{cc}
      \xrightarrow[]{\text{inclusion}}
      \Bialg_{cc}
      \xrightarrow[]{\text{forget}}
      \Coalg_{cc}.
    \end{equation*}
  \end{itemize}
  The only claim left to verify, then, is that a morphism $f:H\to H'$ of cocommutative coalgebras is monic in $\Coalg_{cc}$ if and only if it is so in $\Coalg$.

  $f$ is monic in $\Coalg$ if and only if the comultiplication $\Delta_{H}$ induces an isomorphism
  \begin{equation*}
    H\xrightarrow[\cong]{\quad\Delta_H\quad}H\square_{H'}H
  \end{equation*}
  \cite[Theorem 3.5]{nt_torsion}. Note, though, that that cotensor product is nothing but the self-pullback of $f$ in $\Coalg_{cc}$ (very similar to the remark that the products in $\Coalg_{cc}$ are the tensor products \cite[Theorem 6.4.5]{swe}), so said condition is {\it also} equivalent to $f$ being monic in $\Coalg_{cc}$: it is a formal consequence of the definition that for a morphism $f:c\to c'$ in any category $\cC$ the following conditions are equivalent:
  \begin{itemize}
  \item $f$ is monic;
  \item the diagram
    \begin{equation*}
      \begin{tikzpicture}[auto,baseline=(current  bounding  box.center)]
        \path[anchor=base] 
        (0,0) node (l) {$c$}
        +(2,.5) node (u) {$c$}
        +(2,-.5) node (d) {$c$}
        +(4,0) node (r) {$c'$}
        ;
        \draw[->] (l) to[bend left=6] node[pos=.5,auto] {$\scriptstyle \id$} (u);
        \draw[->] (u) to[bend left=6] node[pos=.5,auto] {$\scriptstyle f$} (r);
        \draw[->] (l) to[bend right=6] node[pos=.5,auto,swap] {$\scriptstyle \id$} (d);
        \draw[->] (d) to[bend right=6] node[pos=.5,auto,swap] {$\scriptstyle f$} (r);
      \end{tikzpicture}
    \end{equation*}
    is a pullback;
  \item the self-pullback $c\times_{c'}c$ exists and either one of the two morphisms $c\times_{c'}c\to c$ is an isomorphism;
  \item the self-pullback $c\times_{c'}c$ exists and the two morphisms $c\times_{c'}c\to c$ are equal;
  \item the self-pullback $c\times_{c'}c$ exists and the diagonal map $c\to c\times_{c'}c$ is an isomorphism. 
  \end{itemize}
\end{proof}

In reference to the last part of the proof of \Cref{pr:commnodiff}, recall \cite[Theorem 3.6]{nt_torsion} that in $\Coalg_{cc}$ the monomorphisms are precisely the injections. Slightly more is true:

\begin{lemma}\label{le:cocomono}
  A $\Coalg$-monomorphism with cocommutative codomain is an embedding. 
\end{lemma}
\begin{proof}
  As in the proof of \cite[Theorem 3.6]{nt_torsion}, the claim reduces to the case when all coalgebras in sight are finite-dimensional. But then one can dualize the statement to address (epi)morphisms of finite-dimensional algebras, whence the conclusion by \cite[\S 2.11]{isb_epidom-4}: an algebra admitting a finite-dimensional epimorphic extension contains a copy of the $2\times 2$ upper-triangular matrices, so cannot be commutative.
\end{proof}

\Cref{le:cocomono} does indeed imply \cite[Theorem 3.6]{nt_torsion}, since by \Cref{pr:commnodiff} \Cref{item:5} a morphism in $\Coalg_{cc}$ is monic if and only if it is so in $\Coalg$. 

\begin{remark}\label{re:epico-monoalg}
  By contrast to the interesting algebra epimorphisms and coalgebra monomorphisms in focus here, the algebra {\it mono}morphisms and coalgebra {\it epi}morphisms are precisely the injective (respectively surjective) morphisms in the relevant category. This is because the forgetful functors
  \begin{equation*}
    \Alg\to \cat{Vec}
    \quad\text{and}\quad
    \Coalg\to \cat{Vec}
  \end{equation*}
  are faithful and right (respectively left) adjoint \cite[\S\S 2.6, 2.7]{porst_bimon}, so they preserve and reflect monomorphisms (respectively epimorphisms)  \cite[Proposition 18.6]{ahs}.   
\end{remark}

Since it is, furthermore, occasionally convenient to extend scalars, it is worth noting that the properties of interest are invariant under such extensions (or contractions).

\begin{lemma}\label{le:fieldext}
  Let $\Bbbk\subset \Bbbk'$ be a field extension. The property of a morphism $f$ of algebras, coalgebras, bialgebras or Hopf algebras over $\Bbbk$ of being monic or epic is equivalent to the corresponding property for the scalar extension $f\otimes_{\Bbbk}\Bbbk'$. 
\end{lemma}
\begin{proof}
  By \Cref{pr:commnodiff} the claims about bialgebras or Hopf algebras reduce to (co)algebras, so we focus on these.
  
  For algebras being monic is equivalent to injectivity: a surjective morphism $A\to B$ is the coequalizer of the two maps $A\times_BA\to A$; dually, for coalgebras being epic is equivalent to surjectivity \cite[Theorem 3.1]{nt_torsion}. We are thus reduced to considering
  \begin{itemize}
  \item epimorphisms $A\to B$ of algebras, equivalent to the multiplication map $B\to B\otimes_AB$ being isomorphisms \cite[\S X.1, Proposition 1.2]{stens_quot};
  \item and monomorphisms $C\to D$ of coalgebras, equivalent to the coproduct map $C\to C\square_D C$ being an isomorphism \cite[Theorem 3.5]{nt_torsion}, recalling the {\it cotensor product} \cite[Definition 8.4.2]{montg_hopf}
    \begin{equation*}
      V\square_D W:=\{x\in V\otimes W\ |\ (\rho_V\otimes\id_W)x = (\id_V\otimes\rho_W)x\in V\otimes D\otimes W\}
    \end{equation*}
    of right and left comodules
    \begin{equation*}
      \rho_V:V\to V\otimes C
      \quad\text{and}\quad
      \rho_W:W\to C\otimes W.
    \end{equation*}
  \end{itemize}
  Being a {\it faithfully flat} \cite[\S 4I]{lam_mod-rng} extension, $\Bbbk'\otimes_{\Bbbk}-$ preserves and reflects everything in sight: (c)tensor products, exactness, surjectivity, injectivity, etc.
\end{proof}

In the sequel, for a right comodule $V$ over a bialgebra $H$, we write
\begin{equation*}
  V^H:=\{v\in V\ |\ v_0\otimes v_1=v\otimes 1\in V\otimes H\}
\end{equation*}
for its space of {\it $H$-coinvariants}. 

\begin{theorem}\label{th:monohopf}
  For a morphism $\pi:H\to H'$ in $\Coalg$ the following conditions are equivalent:
  \begin{enumerate}[(a)]
  \item\label{item:monic} $\pi$ is monic;
  \item\label{item:full} The corestriction functor induced by $\pi$ between categories of (finite-dimensional) comodules is full.
  \item\label{item:idemp} For any (finite-dimensional) $H$-comodule $V$ every idempotent morphism on $\cM^{H'}$ is in $\cM^H$.
  \item\label{item:dsum} For any (finite-dimensional) $H$-comodule $V$ every direct-sum decomposition $V\cong V_1\oplus V_2$ as $H'$-comodules is also one of $H$-comodules.

    If furthermore $\pi$ is a morphism in ${}_{\square}\HAlg$, these are also equivalent to
  \item\label{item:coinvall} For any (finite-dimensional) right $H$-comodule $V\in \cM^H$, the $H'$-coinvariants are also $H$-coinvariants.
  \item\label{item:coinvh} The only $H'$-coinvariants in $H$ are the scalars.
  \end{enumerate}  
  
\end{theorem}
\begin{proof}
  \begin{enumerate}[]
  \item {\bf (Passage to Hopf algebras)} We know from \Cref{pr:commnodiff} \Cref{item:5} that a morphism of Hopf algebras is monic if and only if it is monic in the category of coalgebras.

  \item {\bf \Cref{item:monic} $\Longleftrightarrow$ \Cref{item:full}} For arbitrary comodules, this is part of \cite[Theorem 3.5]{nt_torsion} (or its expansion in \cite[Theorem 2.1]{agore_mono}). Because comodules (of a coalgebra over a field) are unions of finite-dimensional subcomodules \cite[Theore 5.1.1]{mon}, the fullness of $\cM^H\to \cM^{H'}$ is equivalent to the fullness of the analogous functor $\cM^H_f\to \cM^{H'}_f$ between categories of {\it finite-dimensional} comodules.

    This same argument will allow us, in general, to obtain the mutual equivalence between the various statements on arbitrary comodules and their finite-dimensional versions; we will thus not belabor that point.

  \item {\bf \Cref{item:full} $\Longrightarrow$ \Cref{item:idemp}} This is obvious: just apply \Cref{item:full} to the idempotent morphism in question.
    
  \item {\bf \Cref{item:idemp} $\Longleftrightarrow$ \Cref{item:dsum}} A direct-sum decomposition $V\cong V_1\oplus V_2$ is nothing but a choice of idempotent endomorphism on $V$ in the relevant category (with range $V_1$ and kernel $V_2$), so the two are plainly alternative phrasings of a common condition.

  \item {\bf \Cref{item:dsum} $\Longrightarrow$ \Cref{item:full}} Consider $V,W\in \cM^H$, and a linear map $f:V\to W$, $H'$-colinear. The graph
    \begin{equation*}
      \Gamma_f:=\{(v,fv)\ |\ v\in V\} \le V\oplus W
    \end{equation*}
    is then an $H'$-subcomodule, with supplementary summand $W$. Assumption \Cref{item:dsum} then ensures that $\Gamma_f$ is also an $H$-subcomodule of $V\oplus W$, whence the conclusion that $f$ is also $H$-colinear.

    Assume henceforth that $\pi$ is a Hopf-algebra morphism.

  \item {\bf \Cref{item:full} $\Longleftrightarrow$ \Cref{item:coinvall}} Since for $H$-comodules $V$ we have a functorial identification
    \begin{equation*}
      V^H\cong \cM^H(\Bbbk,V),
    \end{equation*}
    fullness implies \Cref{item:coinvall}.

    To verify the converse, note that for finite-dimensional comodules $V,W\in \cM^H_f$ we have
    \begin{equation*}
      \cM^H(V,W)\cong (W\otimes V^*)^H,
    \end{equation*}
    i.e. the $H$-coinvariants of the tensor product of $W$ and the dual of $V$. $H$-comodule morphisms can thus be recast as coinvariants, so \Cref{item:coinvall} does indeed imply \Cref{item:full}.

  \item {\bf \Cref{item:coinvall} $\Longleftrightarrow$ \Cref{item:coinvh}} The latter is nothing but \Cref{item:coinvall} applied to the $H$-comodule $H$ and thus formally weaker than that condition. Conversely, it is enough to recall \cite[Proposition 2.4.3]{dnr} that every $H$-comodule can be embedded in a direct sum of copies of $H$.  \qedhere
  \end{enumerate}
\end{proof}

\Cref{cor:epialggp} confirms, for algebraic groups, some speculative remarks made in the context of studying epimorphisms of {\it Lie algebras} on \cite[pp.13-14]{bg_epi-lie}; see also \cite[Th\'eor\`eme 1]{bb1} for connected algebraic groups and \cite[Theorem 1]{brion_epi} for the general statement (to which \Cref{th:monohopf} specializes when the Hopf algebras in question are commutative).

\begin{corollary}\label{cor:epialggp}
  A closed embedding $\bH\le \bG$ is an epimorphism in the category of linear algebraic groups over a field if and only if the only the only $\bH$-invariants in $\cO(\bG)$ are the scalars. 
\end{corollary}
\begin{proof}
  This is an immediate application of \Cref{th:monohopf}, noting that a closed embedding $\bH\le \bG$ of linear algebraic groups will be an epimorphism precisely when the dual surjection $\cO(\bG)\to \cO(\bH)$ of Hopf algebras of regular functions is a monomorphism.
\end{proof}

\cite[Corollary 2.2]{chi_epi} provides examples of surjective, non-injective monomorphisms of Hopf algebras of the form
\begin{equation*}
  H\xrightarrow[]{\quad\text{antipode}\quad}H^{op,cop}
\end{equation*}
whenever the antipode is surjective but non-injective (e.g. as in \cite[Theorem 3.2]{schau_counter}), but nothing as exotic as that is needed to exhibit such pathological monomorphisms: 

\begin{corollary}\label{cor:parab}
  Let $\bG$ be a connected linear algebraic group over a field and $\bP\le \bG$ a parabolic subgroup. 

  The corresponding surjection $\cO(\bG)\to \cO(\bP)$ of algebras of regular functions is a monomorphism in the category of Hopf algebras.
\end{corollary}
\begin{proof}
  By the very definition of parabolic subgroups \cite[\S 11.2]{brl}, the homogeneous space $\bG/\bP$ is projective. It then follows \cite[Theorem I.3.4]{Hart} that it admits no non-constant regular functions, meaning that condition \Cref{item:coinvh} of \Cref{th:monohopf} is met.
\end{proof}


\begin{remark}\label{re:charpok}
  
  Note that \Cref{cor:parab} goes through in positive characteristic; contrast this with the situation for Lie algebras: by \cite[Theorems 2.2 and 2.3]{bg_epi-lie}, in positive characteristic ($p$, say) epimorphisms of both finite-dimensional Lie algebras and finite-dimensional {\it $p$-Lie algebras} (the {\it restricted Lie algebras of characteristic $p$} of \cite[\S V.7, Definition 4]{jc}) are surjective.
\end{remark}

At the other extreme, so to speak, from epimorphically-embedded (linear algebraic) closed subgroups $\bH\le \bG$ lie the {\it observable} subgroups of \cite[\S 4]{bhm_ext}: those for which the $\bH$-invariants in $\cO(\bG)$ separate the points of the homogeneous space $\bG/\bH$ (as opposed to being constant, as in the epimorphic case).

The connection to epimorphisms was pursued in \cite[Proposition 1]{bb1}, with extensions in \cite[Theorem 5]{brion_epi}: every closed linear algebraic subgroup $\bH\le \bG$ admits a smallest intermediate observable subgroup $\bH\le \overline{\bH}\le \bG$, its {\it observable hull (or envelope)}, which is at the same time the largest subgroup in which $\bH$ embeds epimorphically.

The notion has been developed further in a number of directions: \cite[Theorem 5]{brion_epi} in fact applies to possibly non-affine group schemes, \cite[Theorem 1]{pet_epi} gives a geometric characterization of observability, \cite[Theorem 2.1]{gross_homog} and \cite[Theorem 9]{bt_rat} provide numerous other conditions equivalent to observability, etc. etc.

Different fields and contexts suggest different terminology (for observability), which we will adopt here: see for instance \cite[\S XI.1]{stens_quot}, or \cite{bg_epi-lie}, where the term `dominion' is employed extensively, or \cite[14J]{ahs} for the category-theoretic version we recall and adapt here.

\begin{definition}\label{def:dominion}
  Let $f:c\to d$ be a morphism in a category $\cC$. 
  \begin{enumerate}[(1)]
  \item\label{item:dom} Suppose $\cC$ is complete and well-powered.
    \begin{itemize}
    \item The {\it dominion of $f$} is the smallest regular subobject $\overline{c}\lhook\joinrel\xrightarrow{\quad} d$ (i.e. equalizer of a pair of morphisms $d\to \bullet$ \cite[Definition 7.56]{ahs}) through which $f$ factors.
      
    \item $f$ {\it is a dominion} if the map $c\to\overline{c}$ to the dominion is an isomorphism.

    \item $f$ {\it dominates} a subobject $c'\lhook\joinrel\xrightarrow{\quad} d$ if we have a factorization
      \begin{equation*}
        \begin{tikzpicture}[auto,baseline=(current  bounding  box.center)]
          \path[anchor=base] 
          (0,0) node (l) {$c$}
          +(2,.5) node (ul) {$c'$}
          +(4,.5) node (ur) {$\overline{c}$}
          +(6,0) node (r) {$d$}
          ;
          \draw[->] (l) to[bend left=6] node[pos=.5,auto] {$\scriptstyle $} (ul);
          \draw[right hook->] (ul) to[bend left=6] node[pos=.5,auto] {$\scriptstyle $} (ur);
          \draw[right hook->] (ur) to[bend left=6] node[pos=.5,auto] {$\scriptstyle $} (r);
          \draw[->] (l) to[bend right=6] node[pos=.5,auto,swap] {$\scriptstyle f$} (r);
        \end{tikzpicture}
      \end{equation*}

    \end{itemize}

    Dually:

  \item\label{item:codom} Suppose $\cC$ is cocomplete and co-well-powered.
    \begin{itemize}
    \item The {\it codominion of $f$} is the smallest regular quotient object $c\xrightarrowdbl{\quad} \overline{d}$ (i.e. coequalizer of a pair of arrows $\bullet\to c$ \cite[Definition 7.71]{ahs}) through which $f$ factors.

    \item $f$ (or $d$) {\it is a codominion} if the map $\overline{d}\to d$ is an isomorphism.

    \item $f$ {\it dominates} (this seems a better choice than `codominates') a quotient object $c\xrightarrowdbl{\quad} d'$ if we have a factorization
      \begin{equation*}
        \begin{tikzpicture}[auto,baseline=(current  bounding  box.center)]
          \path[anchor=base] 
          (0,0) node (l) {$c$}
          +(2,.5) node (ul) {$\overline{d}$}
          +(4,.5) node (ur) {$d'$}
          +(6,0) node (r) {$d$}
          ;
          \draw[->>] (l) to[bend left=6] node[pos=.5,auto] {$\scriptstyle $} (ul);
          \draw[->>] (ul) to[bend left=6] node[pos=.5,auto] {$\scriptstyle $} (ur);
          \draw[->] (ur) to[bend left=6] node[pos=.5,auto] {$\scriptstyle $} (r);
          \draw[->] (l) to[bend right=6] node[pos=.5,auto,swap] {$\scriptstyle f$} (r);
        \end{tikzpicture}
      \end{equation*}
    \end{itemize}
  \end{enumerate}
  Under the hypotheses on $\cC$, it can be shown (\cite[\S 14.J]{ahs} and its categorical dual) that such arrows do exist and are unique up to the obvious notions of isomorphism.
\end{definition}


With this in place, note that the observable envelope \cite[\S 3]{bb1} of a closed linear algebraic subgroup $\bH\le \bG$ is nothing but the dominion of $\bH\to \bG$ in the appropriate category (linear algebraic groups in \cite{bb1}, arbitrary algebraic groups in \cite[Theorem 5]{brion_epi}, etc.). 

As to Hopf algebras $H$, there is a satisfying correspondence between dominions (in the category $\Alg$) of right coideal subalgebras $A\le H$ and codominions (in the category $\Coalg$) of left module quotient coalgebras $H\twoheadrightarrow C$. We review some of this material, as worked out partially in \cite[\S 1.1]{chi_cos-xv1} (the discussion is not present in later iterations of that preprint or in the published version \cite{chi_cos}).

Throughout, let $H$ be a Hopf algebra over some field $\Bbbk$. The following constructions are central to \cite{tak_rel} and discussed there extensively:

\begin{notation}\label{not:lr}
  \begin{enumerate}[(1)]
  \item For a right $H$-coideal subalgebra $\iota:A\le H$ (i.e. both a subalgebra and a {\it right coideal} \cite[discussion preceding Example 1.6.6]{montg_hopf}), write $r(\iota)$ or slightly abusively $r(A)$ (`$r$' for `right') for the quotient
    \begin{equation*}
      r(\iota):=H/HA^+,\ A^+ = \ker\varepsilon|_{A},\ \varepsilon\text{ being the counit of }H.
    \end{equation*}
    It is not difficult to see that $r(\iota)$ is a left $H$-module quotient coalgebra of $H$ \cite[Proposition 1]{tak_rel}.

  \item There is a kind of reciprocal construction: given a left $H$-module quotient coalgebra $\pi:H\to C$, write $l(\pi)$ or $l(C)$ (`$l$' meaning `left') for
    \begin{equation*}
      l(\pi):={}^CH:= \{h\in H\ |\ \pi(h_1)\otimes h_2 = \pi(1)\otimes h\}. 
    \end{equation*}
    The same \cite[Proposition 1]{tak_rel} shows that $l(\pi)$ is a right coideal subalgebra. 
  \end{enumerate}
  We might occasionally also drop parentheses in applying the operators $l$ and $r$ (e.g. $rA$ in place of $r(A)$).   
\end{notation}
The operators $l$ and $r$ give back-and-forth maps between
\begin{equation}\label{eq:hposets}
  \begin{aligned}
    &H_{\le}:=\{\text{right $H$-coideal subalgebras}\le H\}
      \quad\text{and}\quad\\
    &H_{\twoheadrightarrow}:=\{\text{left $H$-module quotient coalgebras of }H\}.
  \end{aligned}
\end{equation}
We regard these as posets in the obvious fashion, ordering them 
\begin{itemize}
\item by inclusion for subalgebras;
\item and reverse kernel inclusion for quotients.
\end{itemize}

As observed in \cite[\S 1.1]{chi_cos-xv1}, we have

\begin{lemma}\label{le:lrops}
  Let $H$ be a Hopf algebra with its associated posets \Cref{eq:hposets}.
  \begin{enumerate}[(1)]
  \item\label{item:6} The maps $l:H_{\twoheadrightarrow}\to H_{\le}$ and $r:H_{\le}\to H_{\twoheadrightarrow}$ are order-reversing, and form a Galois connection between the two posets in the sense of \cite[\S IV.5, Theorem 1]{mcl}:
    \begin{equation*}
      rA\ge C \iff lC\ge A. 
    \end{equation*}
    
  \item\label{item:7} The compositions
    \begin{equation*}
      lr:H_{\le} \to H_{\le}
      \quad\text{and}\quad
      rl:H_{\twoheadrightarrow}\to H_{\twoheadrightarrow}
    \end{equation*}
    are closure operators on the two respective posets in the sense of \cite[\S VI.1]{mcl} or \cite[Definition 0-3.8]{ghklms_latt}: order-preserving, idempotent, and such that the image of an element dominates that element. 
  \end{enumerate}
\end{lemma}
\begin{proof}
  \Cref{item:6} is immediate from the definitions, whereas \Cref{item:7} is a formal consequence of \Cref{item:6}.
\end{proof}

The relevance of \Cref{le:lrops} to the present discussion lies in relation to (co)dominions: these too provide closure operators on the two posets, and the two types of closure coincide (see \cite[Proposition 1.1.6]{chi_cos-xv1}).

\begin{proposition}\label{pr:sameclosure}
  Let $H$ be a Hopf algebra and $H_{\le}$, $H_{\twoheadrightarrow}$ its posets \Cref{eq:hposets}.
  \begin{enumerate}[(1)]
  \item For every right $H$-coideal subalgebra $A\in H_{\le}$, $lrA$ is precisely the dominion of $A\le H$ in the category of algebras. 
  \item For every left $H$-module quotient coalgebra $C\in H_{\twoheadrightarrow}$, $rlC$ is the codominion of $H\twoheadrightarrow C$ in the category of coalgebras.  \qedhere
  \end{enumerate}
\end{proposition}

The proof of \cite[Proposition 1.1.6]{chi_cos-xv1} relies on having available a number of equivalent characterizations of dominions (and domination) for comodule morphisms, analogous to the algebra-morphism counterparts in \cite[\S XI.1, Proposition 1.1]{stens_quot}. Those characterizations are left unstated in \cite{chi_cos-xv1}, so we revisit the matter here for completeness.

\begin{theorem}\label{th:coalg-dom}
  Let $f:C\to D$ and $\pi:C\twoheadrightarrow D'$ be two morphisms in $\Coalg_{\Bbbk}$, with $\pi$ onto.

  The following conditions are equivalent:
  \begin{enumerate}[(a)]

  \item\label{item:dom} $f$ dominates $\pi$ in the sense of \Cref{def:dominion} \Cref{item:codom}.

  \item\label{item:mor} For any pair $f_i:C'\to C$, $i=1,2$ of coalgebra morphisms with $f f_1=ff_2$ we have $\pi f_1=\pi f_2$.

  \item\label{item:prod} The two compositions
    \begin{equation}\label{eq:cxc}
      \begin{tikzpicture}[auto,baseline=(current  bounding  box.center)]
        \path[anchor=base] 
        (0,0) node (l) {$C\times_D C$}
        +(4,0) node (m) {$C$}
        +(6,0) node (r) {$D'$}
        ;
        \draw[->] (l) to[bend left=6] node[pos=.5,auto] {$\scriptstyle \id_C\times \varepsilon$} (m);
        \draw[->] (l) to[bend right=6] node[pos=.5,auto,swap] {$\scriptstyle \varepsilon\times\id_C$} (m);
        \draw[->] (m) to[bend left=0] node[pos=.5,auto] {$\scriptstyle \pi$} (r);
      \end{tikzpicture}
    \end{equation}
    are equal, where the leftmost term is the pullback in the category of $\Bbbk$-coalgebras and the two parallel left-hand morphisms are the canonical ones.

  \item\label{item:cotens} The two compositions
    \begin{equation}\label{eq:csqc}
      \begin{tikzpicture}[auto,baseline=(current  bounding  box.center)]
        \path[anchor=base] 
        (0,0) node (l) {$C\square_D C$}
        +(4,0) node (m) {$C$}
        +(6,0) node (r) {$D'$}
        ;
        \draw[->] (l) to[bend left=6] node[pos=.5,auto] {$\scriptstyle \id_C\otimes \varepsilon$} (m);
        \draw[->] (l) to[bend right=6] node[pos=.5,auto,swap] {$\scriptstyle \varepsilon\otimes\id_C$} (m);
        \draw[->] (m) to[bend left=0] node[pos=.5,auto] {$\scriptstyle \pi$} (r);
      \end{tikzpicture}
    \end{equation}
    are equal, where the leftmost term is the cotensor product \cite[Definition 8.4.2]{montg_hopf}.

  \item\label{item:x} For any $C$-bicomodule $X$ and element $x\in X$, we have
    \begin{equation}\label{eq:x01f}
      x_0\otimes f(x_1) = x_0\otimes f(x_{-1})\in X\otimes D
      \Longrightarrow
      x_0\otimes \pi(x_1) = x_0\otimes \pi(x_{-1})\in X\otimes D'.
    \end{equation}

  \item\label{item:x-fin} The same condition, but only for {\it finite-dimensional} bicomodules $X$.

  \item\label{item:vtow} Given a linear map $\varphi:V\to W$ between (right, say) $C$-comodules, if $\varphi$ is a morphism of $D$-comodules via corestriction along $f$ then it is a morphism of $D'$-comodules via corestriction along $\pi$:
    \begin{equation*}
      \begin{aligned}
        (\id_W\otimes f)\circ\rho_W\circ \varphi &= (\varphi\otimes f)\circ\rho_V
                                                   \text{ implies }\\
        (\id_W\otimes \pi)\circ\rho_W\circ \varphi &= (\varphi\otimes \pi)\circ\rho_V,
      \end{aligned}            
    \end{equation*}
    where $\rho_V:V\to V\otimes C$ and $\rho_W:W\to W\otimes C$ are the comodule structure maps.

  \item\label{item:vtow-fin} As above, but only for {\it finite-dimensional} $C$-comodules $V$ and $W$.

  \end{enumerate}
\end{theorem}
\begin{proof}
  \begin{enumerate}[]
  \item {\bf \Cref{item:x} $\iff$ \Cref{item:x-fin} and \Cref{item:vtow} $\iff$ \Cref{item:vtow-fin}} This follows immediately from the fact that comodules are unions of finite-dimensional subcomodules \cite[Theorem 5.1.1]{mon}. This holds for bicomodules too of course: a $C$-bicomodule is nothing but a comodule over $C\otimes C^{cop}$,where $C^{cop}$ is the coalgebra with ``co-opposite'' comultiplication
    \begin{equation*}
      C\ni c\xmapsto{\quad} c_2\otimes c_1\in C\otimes C.
    \end{equation*}
    We will thus only have to consider finite-dimensional (bi)comodules.

  \item {\bf \Cref{item:dom} $\iff$ \Cref{item:prod}} The codominion of $f$ is the coequalizer of the two parallel left-hand maps in \Cref{eq:cxc} \cite[14J (g)]{ahs}, hence the equivalence.


  \item {\bf \Cref{item:mor} $\iff$ \Cref{item:prod}}. On the one hand, the former is formally stronger because the two parallel arrows in \Cref{eq:cxc} can serve the role of $f_i$ in \Cref{item:mor}. On the other, $f_i$ as in \Cref{item:mor} will have to factor uniquely through \Cref{item:prod} by the universality property of the pullback. This gives the opposite implication $\Cref{item:prod}\Rightarrow \Cref{item:mor}$ and hence the equivalence.

  \item {\bf \Cref{item:prod} $\iff$ \Cref{item:cotens}} For any pair of morphisms $f_i:C'\to C$, $i=1,2$ with $f f_1=f f_2$ we have a map
    \begin{equation*}
      C'\ni c\xmapsto{\quad}f_1(c_1)\otimes f_2(c_2)\in C\square_D C\le C\otimes C.
    \end{equation*}
    In particular, there is a canonical map $C\times_D C\to C\square_D C$ which together with \Cref{eq:csqc} factors \Cref{eq:cxc}. Note next, as a simple exercise, that the coequalizer (in $\cat{Vec}$!) of the two left-hand parallel arrows in \Cref{eq:csqc} is already a quotient {\it coalgebra} of $C$.

    It follows that the codominion of $f$ can also be recovered as the coequalizer in $\cat{Vec}$ of the two left-hand maps in \Cref{eq:csqc}, so the conclusion follows.

  \item {\bf \Cref{item:x-fin} $\Longrightarrow$ \Cref{item:vtow-fin}} Given $V,W\in \cM^C_f$ (hence finite-dimensional), apply \Cref{item:x-fin} to the finite-dimensional $C$-bicomodule
    \begin{equation*}
      X:=\Hom(V,W)\cong W\otimes V^*
    \end{equation*}
    with its right comodule structure inherited from $W$ and the left structure inherited from that on $V^*$ (which in turn results from dualizing the original {\it right} comodule structure on $V$).

  \item {\bf \Cref{item:vtow} $\Longrightarrow$ \Cref{item:mor}} This will be entirely analogous to the proof of \cite[\S XI.1, Proposition 1.1, implication (d) $\Rightarrow$ (a)]{stens_quot}.

    Morphisms $f_i:C'\to C$ as in \Cref{item:mor} will induce right $C$-comodule structures on $C'$ so that $f_1$ becomes a $D$-comodule morphism
    \begin{itemize}
    \item from $C'_{ff_2}$, meaning $C'$ equipped with the $D$-comodule structure induced by $ff_2$ (which is equal to $ff_1$ by hypothesis);
    \item to $C_f$, meaning $C$ equipped with the $D$-comodule structure induced by $f$.
    \end{itemize}
    By assumption, the same goes with $\pi$ in place of $f$:
    \begin{equation*}
      \cM^{D'}\ni C'_{\pi f_2}\xrightarrow{\quad f_1\quad}C_{\pi}\in \cM^{D'}.
    \end{equation*}
    This, though, means precisely that $\pi f_1=\pi f_2$. 
    
  \item {\bf \Cref{item:mor} $\Longrightarrow$ \Cref{item:x-fin}} The argument is very similar to the proof of \cite[\S XI.1, Proposition 1.1, implication (a) $\Rightarrow$ (b)]{stens_quot}.
    
    Consider a finite-dimensional $C$-bicomodule $X$ with right and left comodule structures denoted as usual, by
    \begin{equation*}
      \begin{aligned}
        X\ni x&\xmapsto{\quad}x_0\otimes x_1\in X\otimes C\quad\text{and}\\
        X\ni x&\xmapsto{\quad}x_{-1}\otimes x_0\in C\otimes X\text{ respectively}.
      \end{aligned}
    \end{equation*}
    This equips the dual vector space $X^*$ with a $C$-bicomodule structure in the only sensible fashion: denoting the generic element of $X^*$ by $x^*$, we have
    \begin{equation*}
      x^*_0(x)x^*_1 = x^*(x_0)x_{-1}\in C
      \quad\text{and}\quad
      x^*_0(x)x^*_{-1} = x^*(x_0)x_{1}\in C
    \end{equation*}
    for all $x^*\in X$ and $x\in X$. One can then equip the space $C':=C\oplus X^*$ with a coalgebra structure by keeping the comultiplication and counit we already have on $C$, and setting
    \begin{equation*}
      X^*\ni x^*\xmapsto{\quad\Delta\quad} x^*_0\otimes x^*_1+x^*_{-1}\otimes x^*_0
      \quad\in\quad X^*\otimes C\oplus C\otimes X^*\le C'\otimes C'
    \end{equation*}
    and $\varepsilon|_{X^*}\equiv 0$. 

    Now fix an $x\in X$ satisfying the hypothesis of \Cref{item:x} (i.e. the left-hand side of \Cref{eq:x01f}), and consider the two coalgebra morphisms $f_i:C'\to C$, $i=1,2$ defined by
    \begin{itemize}
    \item $f_1:=$ obvious projection of $C'=C\oplus X^*$ onto $C$;
    \item $f_2|_{C}:=\id_C$;
    \item and
      \begin{equation*}
        X^*\ni x^*\xmapsto{\quad f_2\quad}x^*(x_0)x_1-x^*(x_0)x_{-1} = x^*_0(x)x^*_{-1}-x^*_0(x)x^*_{1}\in C. 
      \end{equation*}
    \end{itemize}
    We have $ff_1=ff_2$ and hence also $\pi f_1=\pi f_2$ by assumption, whence right-hand side of \Cref{eq:x01f}.
  \end{enumerate}  
  This completes the proof. 
\end{proof}

The discussion above focuses on dominions of (co)algebra morphisms, whereas we will later be interested in the richer structures mentioned before (bialgebras, Hopf algebras). It is good to know, then, that a result analogous to \Cref{pr:commnodiff} holds: it doesn't make much difference where one computes the (co)dominion, so long as it makes sense to do so.

\begin{theorem}\label{th:same-co-dom}
  \begin{enumerate}[(1)]
  \item\label{item:samedom} The dominion of an algebra morphism $f:H\to H'$ is the same in any of the categories
    \begin{equation*}
      \Alg_{\bullet},\ \Bialg_{\bullet},\ \HAlg_{\bullet},\ {}_{bi}\HAlg,\ {}_{2}\HAlg
    \end{equation*}
    that happen to contain it.
  \item\label{item:samecodom} The codominion of a coalgebra morphism $f:H\to H'$ is the same in any of the categories
    \begin{equation*}
      \Coalg_{\bullet},\ \Bialg_{\bullet},\ \HAlg_{\bullet},\ {}_{bi}\HAlg,\ {}_{2}\HAlg
    \end{equation*}
    that happen to contain it.
  \end{enumerate}
\end{theorem}
\begin{proof}
  The two statements are once more mutual duals, with one proof easily adapted into the other, so we focus on \Cref{item:samecodom}.

  Some of the discussion bifurcates over cocommutativity. For any of the categories in question, the codominion of $f$ is the coequalizer of the two maps
  \begin{equation}\label{eq:h'h'h'}
    \begin{tikzpicture}[auto,baseline=(current  bounding  box.center)]
      \path[anchor=base] 
      (0,0) node (l) {$H\times_{H'}H$}
      +(2,0) node (r) {$H$}
      ;
      \draw[->] (l) to[bend left=6] node[pos=.5,auto] {$\scriptstyle $} (r);
      \draw[->] (l) to[bend right=6] node[pos=.5,auto] {$\scriptstyle $} (r);
    \end{tikzpicture}
  \end{equation}
  from the pullback in the corresponding category (by \cite[14J (g)]{ahs}, as in the proof of the equivalence \Cref{item:dom} $\iff$ \Cref{item:prod} in \Cref{th:coalg-dom}). Those pullbacks range over just two possibilities: those computed in $\Coalg_{cc}$ and those computed in $\Coalg$; this is because the relevant forgetful functors to $\Coalg_{\bullet}$ are right adjoints, once more by \cite[diagram (9)]{porst_formal-2} and \cite[\S 2.2]{porst_univ}. The claim, then, amounts to this:

  For $f:H\to H'$ in any of the categories of bialgebras or Hopf algebras in \Cref{item:samecodom}, the forgetful functor to $\Coalg_{\bullet}$ preserves coequalizers of pairs of the form \Cref{eq:h'h'h'}.

  Even more specifically, $\Coalg_{\bullet}\to\cat{Vec}$ is a left adjoint and hence cocontinuous \cite[\S 4.3]{porst_bimon}, we are claiming that the coequalizers of the form \Cref{eq:h'h'h'} in any of the categories $\Bialg_{\bullet}$ or ${}_{\square}\HAlg_{\bullet}$ are computable in $\cat{Vec}$.

  To see this, note that the pair \Cref{eq:h'h'h'} is {\it reflexive} \cite[\S 3.3, preceding Proposition 4]{bw}: the two arrows have a common right inverse $H'\to H'\times_HH'$ whose components are both $\id_{H'}$. The conclusion then follows from \Cref{le:reflok}.
\end{proof}

\begin{lemma}\label{le:reflok}
  \begin{enumerate}[(1)]
  \item\label{item:ordeq} The forgetful functors from $\Bialg_{\bullet}$ or ${}_{\square}\HAlg_{\bullet}$ to $\cat{Vec}$ preserve reflexive equalizers, with `$\bullet$' meaning any combination of `c' or `cc' and `$\square$' either blank, `bi' or `2'.
  \item\label{item:ordcoeq} Dually, the forgetful functors from $\Bialg_{\bullet}$ or ${}_{\square}\HAlg_{\bullet}$ to $\cat{Vec}$ preserve reflexive coequalizers.
  \end{enumerate}
\end{lemma}
\begin{proof}
  Singling out one of the mutually-dual claims again, say \Cref{item:ordcoeq}:

  Let $f_i:X\to Y$, $i=1,2$ be a reflexive pair of parallel morphisms in any of the categories in question, with a common right inverse $s:Y\to X$: $f_1s=f_2s=\id_Y$. Note now that the kernel
  \begin{equation}\label{eq:f1-f2}
    \{f_1(x)-f_2(x)\ |\ x\in X\}\le Y
  \end{equation}
  of the vector-space coequalizer $Y\to \overline{Y}$ of $(f_i)$ is both a coideal (naturally, since colimits in $\Coalg$ coincide with those in $\cat{Vec}$) and, by the reflexivity assumption, an ideal:
  \begin{equation*}
    (f_1(x)-f_2(x))y = f_1(xs(y)) - f_2(xs(y)),\ \forall x\in X,\ y\in Y,
  \end{equation*}
  and similarly for left multiplication. It follows that the vector-space coequalizer is in fact a quotient bialgebra, and invariance of \Cref{eq:f1-f2} under antipodes (in $\HAlg$) or inverse antipodes (in ${}_{bi}\HAlg$, etc.) is immediate.
\end{proof}

\begin{remark}\label{re:beknotenough}
  The proof of \Cref{th:same-co-dom} touches on the issue of cokernel preservation for the various forgetful functors
  \begin{equation*}
    \Bialg_{\bullet},\ {}_{\square}\HAlg_{\bullet}\xrightarrow{\quad}\Coalg_{\bullet}. 
  \end{equation*}
  Since these functors are in fact {\it monadic} (or {\it tripleable} \cite[\S 3.3]{bw}) by \cite[Theorem 10 and Proposition 28]{porst_formal-2} and \cite[\S 4.3]{porst_bimon}. This entails the preservation of {\it some} reflexive coequalizers ({\it Beck's Theorem}: \cite[\S 3.3, Theorem 10]{bw}) but not, generally, all of them.
  
  As a matter of fact, \Cref{le:reflok} could have been obtained in a more roundabout (and conceptual) manner by tracing through the monadicity proof in \cite{porst_bimon}; the more direct route seemed preferable in this specific case. 
\end{remark}

There is also an analogue, for (co)dominions, of \Cref{le:fieldext}: field extensions preserve the relevant constructions.

\begin{proposition}\label{pr:fieldextcodom}
  The scalar-extension functor $\Bbbk'\otimes_{\Bbbk}-$ along a field extension preserves (co)dominions in all of the categories $\Alg_{\bullet}$, $\Coalg_{\bullet}$, $\Bialg_{\bullet}$ and ${}_{\square}\HAlg_{\bullet}$.
\end{proposition}
\begin{proof}
  \Cref{th:same-co-dom} reduces the problem to $\Alg$ and $\Coalg$. Moreover, since {\it codominions} of algebra morphisms and {\it dominions} of coalgebra morphisms are uninteresting (they are just the images of the respective morphisms), we discuss dominions in $\Alg$ and codominions in $\Coalg$.

  We saw in the course of the proof of \Cref{th:coalg-dom} that the codominion of $C\to D$ is the coequalizer (in $\vc$) of the two maps $C\square_DC\to C$. Dually, the dominion of an algebra morphism $A\to B$ is the equalizer of the two maps $B\to B\otimes_AB$ \cite[\S XI.1, Proposition 1.1]{stens_quot}. As in the proof of \Cref{le:fieldext}, field scalar extensions preserve all of these constructions.
\end{proof}

\begin{remarks}\label{res:fieldextcocont}
  \begin{enumerate}[(1)]
    
  \item See \cite[Proposition 5.3]{bg_epi-lie} for an analogue of \Cref{pr:fieldextcodom} for finite-dimensional Lie algebras.
    
  \item A version of \Cref{pr:fieldextcodom} holds for algebraic groups \cite[Theorem 5 (ii)]{brion_epi}: the latter's specialization to {\it linear} algebraic groups amounts precisely to the $\HAlg_c$-codominion instance of \Cref{pr:fieldextcodom}.

  \end{enumerate}
\end{remarks}

\Cref{th:coalg-dom} helps relate codominions to the notion of {\it coflatness}. Recall \cite[\S 10.8]{brz-wis} that a left $C$-comodule $M$ is
\begin{itemize}
\item {\it coflat} if the functor $\cM^C\xrightarrow{\quad\square_CM\quad}\vc$ is exact;
\item and {\it faithfully coflat} if that functor is in addition faithful.
\end{itemize}
We apply the term to coalgebra morphisms $C\to D$: they are right (faithfully) coflat if $C\in \cM^D$ is (faithfully) coflat, and similarly on the left. 

\begin{remark}\label{re:coflatinj}
  As it happens, for coalgebras coflatness and injectivity interact very conveniently:
  \begin{itemize}
  \item coflatness is equivalent to injectivity in the relevant category of comodules (\cite[Theorem 2.4.17]{dnr} or \cite[\S 10.12, (ii) (1)]{brz-wis});
  \item and faithful coflatness is equivalent \cite[\S 10.12, (ii) (2)]{brz-wis} to the comodule in question being an {\it injective cogenerator} (recall \cite[\S V.7]{mcl} that a cogenerator in a category $\cC$ is an object $d$ such that whenever $\cC$-morphisms $f_i:c\to c'$, $i=1,2$ differ so do $gf_i$ for some $g:c'\to d$).
  \end{itemize}
\end{remark}

One can now replicate (and dualize) the usual discussion of {\it faithfully flat descent} (\cite[Teorema]{cip_ff}, \cite[Theorem 3.8]{nuss_ff}) for categories of modules, in its formulation in terms of categories of comodules over {\it corings} \cite[\S 25.4]{brz-wis}.

\begin{construction}
  For any coalgebra $C$
  \begin{itemize}
  \item the category ${}^C\cM^C$ of $C$-bicomodules is {\it monoidal} (\cite[\S VII.1]{mcl}, \cite[Definition 2.1.1]{egno_tensor}), with $\square_{C}$ as the tensor bifunctor and $C$ as the monoidal unit.
  \item It thus makes sense to speak of {\it algebras internal to ${}^C\cM^C$} \cite[Definition 7.8.1]{egno_tensor} and categories of {\it modules} over them \cite[Definition 7.8.5]{egno_tensor}
  \item Similarly, for a unital internal algebra $A\in {}^C\cM^C$, we denote by $\cM_A$ the category (of {\it right $A$-modules}) consisting of right $C$-comodules $V\in \cM^C$ equipped with a morphism
    \begin{equation*}
      V\square_CA\to V\quad\text{in}\quad \cM^C,
    \end{equation*}
    associative and unital in the obvious sense.
    
    The difference to \cite[Definition 7.8.5]{egno_tensor} lies in $V$ being only a right $C$-comodule rather than a $C$-{\it bi}comodule.
  \end{itemize}
  Consider, now, a coalgebra morphism $C\to D$.
  \begin{itemize}
  \item The cotensor product $C\square_DC$ is a unital algebra in ${}^C\cM^C$, with multiplication
    \begin{equation*}
      (C\square_DC)\square_C(C\square_DC)\cong C\square_DC\square_DC\xrightarrow{\quad\id\otimes\varepsilon_C\otimes\id} C\square_DC
    \end{equation*}
    and unit
    \begin{equation*}
      C\xrightarrow{\quad\Delta_C\quad} C\square_DC.
    \end{equation*}
  \item For any (right) $D$-comodule $W$, the cotensor product $W\square_DC$ is not only a right $C$-comodule, but also a right $C\square_DC$-module with multiplication
    \begin{equation*}
      (W\square_DC)\square_C(C\square_DC)\cong W\square_DC\square_DC\xrightarrow{\quad\id\otimes\varepsilon_C\otimes\id \quad} W\square_DC.
    \end{equation*}

  \item This gives a functor
    \begin{equation*}
      \cM^D\xrightarrow{\quad -\square_DC\quad}\cM_{C\square_DC},
    \end{equation*}
    the right half of an adjunction
    \begin{equation}\label{eq:fcadj}
      \begin{tikzpicture}[auto,baseline=(current  bounding  box.center)]
        \path[anchor=base] 
        (0,0) node (l) {$\cM_{C\square_DC}$}
        +(2,0) node (m) {$\bot$}
        +(4,0) node (r) {$\cM^D$,}
        ;
        \draw[->] (l) to[bend left=20] node[pos=.5,auto] {$\scriptstyle N\xmapsto{\quad} N_{C\square_DC}$} (r);
        \draw[->] (r) to[bend left=20] node[pos=.6,auto] {$\scriptstyle -\square_DC$} (l);
      \end{tikzpicture}
    \end{equation}
    where the space $N_{C\square_DC}$ of {\it coinvariants} (by analogy to the group-homology term \cite[\S 1.1]{ev_coh}) of the $C\square_DC$-module $N$ is defined as the (vector-space) coequalizer
    \begin{equation}\label{eq:coeqn}
      \begin{tikzpicture}[auto,baseline=(current  bounding  box.center)]
        \path[anchor=base] 
        (0,0) node (l) {$N\square_D C$}
        +(4,0) node (m) {$N$}
        +(6,0) node (r) {$N_{C\square_DC}$}
        ;
        \draw[->] (l) to[bend left=6] node[pos=.5,auto] {$\scriptstyle \text{action}$} (m);
        \draw[->] (l) to[bend right=6] node[pos=.5,auto,swap] {$\scriptstyle \id\otimes\varepsilon$} (m);
        \draw[->] (m) to[bend left=0] node[pos=.5,auto] {$\scriptstyle $} (r);
      \end{tikzpicture}
    \end{equation}
  \end{itemize}
\end{construction}

Having set up all of this, \Cref{th:fcdesc} is a straightforward dualization of the standard arguments for faithfully flat descent \cite[Teorema]{cip_ff}. We give a proof for completeness; as noted in \cite[\S 4.2]{del_tann} though, such results are nowadays easily available as consequences of Beck's \cite[\S 3.3, Theorem 10]{bw} on (co)monadic functors.

\begin{theorem}\label{th:fcdesc}
  let $f:C\to D$ be a coalgebra morphism. The adjunction \Cref{eq:fcadj} is an equivalence in either of the following conditions:
  \begin{itemize}
  \item $f$ is left faithfully coflat;
  \item or $f$ is a split epimorphism (i.e. has a right inverse) in ${}^D\cM^D$.
  \end{itemize}
\end{theorem}
\begin{proof}
  Consider the commutative diagram
   \begin{equation*}
     \begin{tikzpicture}[auto,baseline=(current  bounding  box.center)]
       \path[anchor=base] 
       (0,0) node (l) {$\cM^D$}
       +(4,0) node (r) {$\cM_{C\square_DC}$}
       +(2,-1) node (d) {$\cM^C$}
       ;
       \draw[->] (l) to[bend left=6] node[pos=.5,auto] {$\scriptstyle -\square_DC$} (r);
       \draw[->] (l) to[bend right=6] node[pos=.5,auto,swap] {$\scriptstyle -\square_DC$} (d);
       \draw[->] (r) to[bend left=6] node[pos=.5,auto] {$\scriptstyle \text{forget}$} (d);
       \end{tikzpicture}
   \end{equation*}
   Now observe that
   \begin{itemize}
   \item the left-hand arrow is right adjoint to the scalar corestriction functor $\cM^C\to \cM^D$;
   \item the monad on $\cM^C$ associated to that adjunction \cite[\S 3.1, Theorem 1]{bw} is nothing but $-\square_DC:\cM^C\to \cM^C$;
   \item and the horizontal arrow is the {\it comparison functor} \cite[\S 3.2, following Proposition 2]{bw} associated to that selfsame adjunction (and monad). 
   \end{itemize}
   The condition that the horizontal arrow be an equivalence is now equivalent to the monadicity of $-\square_DC:\cM^D\to \cM^C$, this, in turn holds in the two flagged cases:
   \begin{itemize}
   \item If $C\in {}^D\cM$ is faithfully coflat then the conclusion follows from Beck's theorem (\cite[\S 3.3, Theorem 10]{bw}): we have already observed it is a right adjoint, and it reflects isomorphisms and preserves {\it all} coequalizers (not just suitable reflexive ones) by faithful coflatness.
   \item On the other hand, if $f:C\to D$ has a right inverse in ${}^D\cM^D$, then the counit
     \begin{equation*}
       -\square_DC\to \id\qquad\left(\text{natural transformation of functors}\quad\cM^D\to \cM^D\right)
     \end{equation*}
     of the adjunction between $\cM^D$ and $\cM^C$ has a right inverse, hence the monadicity claim by (the categorical dual of) \cite[Proposition 3.16]{mesabl_desc}.
   \end{itemize}
\end{proof}

To circle back to codominions:

\begin{corollary}\label{cor:coflatcodom}
  A morphism $C\to D$ in any of the categories 
  \begin{equation}\label{eq:coalgcats}
    \Coalg_{\bullet},\quad
    \Bialg_{\bullet}\quad\text{or}\quad
    {}_{\square}\HAlg_{\bullet}
  \end{equation}
  that
  \begin{itemize}
  \item is left or right faithfully coflat;
  \item or is a split epimorphism in ${}^{D}\cM^{D}$
  \end{itemize}  
  is a codominion. 
\end{corollary}
\begin{proof}
  Per \Cref{th:same-co-dom} \Cref{item:samecodom}, it is enough to handle $\Coalg$. That claim, though, follows immediately from the description of the codominion of $C\to D$ as the coequalizer \Cref{eq:csqc} and from \Cref{th:fcdesc}: for
  \begin{equation*}
    N=C\cong D\square_DC\in \cM_{C\square_DC}
  \end{equation*}
  the two parallel left-hand arrows of \Cref{eq:coeqn} specialize precisely to those of \Cref{eq:csqc}.
\end{proof}

\begin{remark}\label{re:recovertak}
  In the context of left module quotient coalgebras of Hopf algebras, fitting into the framework of \Cref{pr:sameclosure}, \Cref{cor:coflatcodom} specializes back to (a particular case of) \cite[Theorem 2]{tak_rel}.
\end{remark}

Recall that the {\it cosemisimple} coalgebras (\cite[Definition 2.4.1 and Lemma 2.4.3]{montg_hopf}, \cite[Theorem 3.1.5]{dnr}) are those whose (left or right) comodules are all injective (or equivalently, projective).

\begin{corollary}\label{cor:ontocoss}
  A surjection in any of the categories \Cref{eq:coalgcats} onto a cosemisimple object is a codominion.
\end{corollary}
\begin{proof}
  This follows from \Cref{cor:coflatcodom} once we observe that for cosemisimple $C$ a (right, say) $C$-comodule $M$ surjecting onto $C$ is faithfully coflat. Indeed, $M$ coflat by assumption, and, $C$ being projective in $\cM^C$ \cite[Theorem 3.1.5]{dnr}, $M$ will also contain it as a summand. But {\it every} coalgebra is an injective cogenerator in its own category right comodules \cite[\S 9.1]{brz-wis}, hence so is $M\in \cM^C$ (so that it is also faithfully coflat \cite[\S 10.12, (ii) (2)]{brz-wis}).
\end{proof}

\begin{remark}\label{re:bg-ss-lalg}
  Compare \Cref{cor:ontocoss} to the remark made in passing in \cite[Addenda]{bg_epi-lie} (and attributed to Hochschild) that every semisimple Lie subalgebra of any finite-dimensional Lie algebra is a dominion in the category of finite-dimensional Lie algebras.

  
  Since by \cite[Theorems 2.2 and 2.3]{bg_epi-lie} that question is only interesting in characteristic zero, where semisimple Lie algebras have semisimple categories of finite-dimensional representations \cite[\S 6.3, Theorem]{hmph_intro}, the connection to cosemisimplicity starts to become apparent. 
\end{remark}

As a matter of fact, the second bullet point in \Cref{cor:coflatcodom} is sub-optimal: {\it one}-sided splitting suffices.

\begin{theorem}\label{th:onesplit}
  The conclusion of \Cref{cor:coflatcodom} holds provided $C\to D$ has a right inverse as a either a left or right $D$-comodule morphism. 
\end{theorem}
\begin{proof}
  The proof of \Cref{cor:coflatcodom} makes it apparent that one does not quite need the adjunction \Cref{eq:fcadj} to be an equivalence: it is enough that its counit be an equivalence (in other words \cite[\S IV.3, Theorem 1]{mcl}: that the right adjoint $-\square_DC$ be fully faithful).

  In other words, fora coalgebra morphism $f:C\to D$ to be a codominion all we need is that
  \begin{equation}\label{eq:splitfork}
    \begin{tikzpicture}[auto,baseline=(current  bounding  box.center)]
      \path[anchor=base] 
      (0,0) node (l) {$N\square_D C\square_DC$}
      +(4,0) node (m) {$N\square_DC$}
      +(6,0) node (r) {$N$}
      ;
      \draw[->] (l) to[bend left=6] node[pos=.5,auto] {$\scriptstyle \id_N\otimes\id_C\otimes\varepsilon$} (m);
      \draw[->] (l) to[bend right=6] node[pos=.5,auto,swap] {$\scriptstyle \id_N\otimes\varepsilon\otimes\id_C$} (m);
      \draw[->] (m) to[bend left=0] node[pos=.5,auto] {$\scriptstyle \id_N\otimes\varepsilon$} (r);
    \end{tikzpicture}
  \end{equation}
  be a coequalizer in $\vc$. In the presence of a left $D$-comodule right inverse $s:D\to C$ to $f$, the maps
  \begin{equation*}
    N\xrightarrow{\quad\id_N\otimes s\quad}N\square_DC
    \quad\text{and}\quad
    N\square_DC\xrightarrow{\quad\id\otimes\id\otimes s\quad} N\square_DC\square_DC
  \end{equation*}
  make \Cref{eq:splitfork} into a {\it split} (or {\it contractible}) {\it coequalizer} \cite[\S 3.3, following Proposition 1]{bw} in $\vc$, so in particular (as the name suggests) a $\vc$-coequalizer \cite[\S 3.3, Proposition 2]{bw}.
\end{proof}

\begin{remark}\label{re:bichonsplit}
  The remark behind the proof of \Cref{th:onesplit} is precisely dual to \cite[Proposition 2.2]{bichon_expect}, concerned with algebra morphisms $A\to B$ split as right $A$-module maps.
\end{remark}

\section{CQG algebras and compact groups}\label{se:cqg}

One class of Hopf algebras particularly well-suited for generalizing compact-group representation theory is that of {\it CQG algebras} \cite[Definition 2.2]{dk_cqg}. One version of that definition is as follows:

\begin{definition}\label{def:cqg-alg}
  A {\it CQG algebra} is
  \begin{itemize}
  \item a complex Hopf $*$-algebra $H$ (\cite[\S 1]{dk_cqg}: complex Hopf algebra equipped with a conjugate-linear multiplication-reversing involution `$*$' such that both $\Delta$ and $\varepsilon$ are $*$-morphisms);

    
  \item all of whose finite-dimensional comodules
    \begin{equation*}
      V\ni v\xmapsto{\quad}x_0\in v_1\in V\otimes H
    \end{equation*}
    are {\it unitarizable}: admitting an inner product $\braket{-\mid -}$ (linear in the second variable, say), compatible with the comodule structure in the sense that
    \begin{equation*}
      \braket{v\mid w_0}Sw_1 = \braket{v_0\mid w}v_1^*,\ \forall v,w\in V. 
    \end{equation*}
    \end{itemize}
    $\cqg$ will denote the category of CQG algebras, with Hopf *-algebra morphisms as its arrows.
\end{definition}

\begin{remark}\label{re:cqg-coss}
  Cosemisimplicity is a consequence of \Cref{def:cqg-alg}: the orthogonal complement of a subcomodule is again a subcomodule, hence the semisimplicity of the category $\cM^H$.
\end{remark}

The motivation stems from the fact that the {\it commutative} CQG algebras are precisely the algebras of {\it representative functions} \cite[\S III.1, Definition 1.1]{bd_lie} on compact groups, and the usual Tannakian duality machinery \cite[\S III.7]{bd_lie} can phrased as a contravariant equivalence between the category $\cqg_c$ of commutative CQG algebras and the category of compact groups \cite[Theorems 2.6 and 2.8]{wang_fp}. For that reason, one thinks of the category $\cqg$ as dual to that of {\it compact quantum groups} (this being one possible definition of the latter category, though not the only one: see e.g. \cite[Definition 2.3 and subsequent discussion]{wang_fp}).

Recall, in this context, that embeddings of compact groups are equalizers \cite[Theorem 2.1]{chi_multipush}, whence

\begin{corollary}\label{cor:cpct}
  The dominion of a compact-group morphism is the embedding of its image in its codomain.

  Dually, the codominion of a morphism in the category $\cqg_c$ of commutative CQG algebras is the surjection of its domain onto its image.  \qedhere
\end{corollary}

This also implies that epimorphisms of compact groups are onto (\cite[Theorem]{pogunt_epi-cpct}, \cite[Proposition 9]{reid-epi}) or equivalently, monomorphisms in $\cqg_c$ are embeddings. This latter result transports over to $\cqg$: the monomorphisms in that category are precisely the injective morphisms \cite[Proposition 6.1]{chi_cat-cqg}.

In view of all of the above, it seems reasonable to ask whether one can strengthen \cite[Proposition 6.1]{chi_cat-cqg} in the same direction, by showing that codominions (rather than just monomorphisms) behave ``as expected'':

\begin{theorem}\label{th:cqgcodom}
  A surjection in $\cqg$ is a coequalizer in that category.
\end{theorem}
\begin{proof}
  Recall \cite[Theorem 3.1]{chi_cat-cqg} that $\cqg$ is locally presentable, so much of the discussion in \Cref{se:prel} applies: the category is complete and cocomplete, etc. In particular, we can freely refer to pullbacks therein.
  
  Let $f:H\to H'$ be a surjective morphism of CQG algebras. As in the proof of the equivalence \Cref{item:prod} $\iff$ \Cref{item:cotens} in \Cref{th:coalg-dom}, the vector-space coequalizer of the two maps
  \begin{equation*}
    \begin{tikzpicture}[auto,baseline=(current  bounding  box.center)]
      \path[anchor=base] 
      (0,0) node (l) {$H\square_{H'}H$}
      +(2,0) node (r) {$H$}
      ;
      \draw[->] (l) to[bend left=6] node[pos=.5,auto] {$\scriptstyle $} (r);
      \draw[->] (l) to[bend right=6] node[pos=.5,auto] {$\scriptstyle $} (r);
    \end{tikzpicture}
  \end{equation*}
  is already a CQG algebra, and hence also the coequalizer of the two maps \Cref{eq:h'h'h'} from the pullback in $\cqg$. The latter is the codominion in $\cqg$ \cite[14J (g)]{ahs}, so that codominion is also computable in $\Coalg$ (or $\Bialg$, etc.). But $H'$ is cosemisimple, so the codominion is $H\to H'$ itself by \Cref{cor:ontocoss}.
\end{proof}

The non-commutative analogue of \Cref{cor:cpct} is now a consequence (or rather rephrasing):

\begin{corollary}\label{cor:cqg}
  the codominion of a morphism in the category $\cqg$ of CQG algebras is the surjection of its domain onto its image.  \qedhere
\end{corollary}

\begin{remark}\label{re:cpctfollows}
  In fact, \Cref{cor:cpct} {\it follows} from \Cref{cor:cqg} (thus giving an alternative proof of \cite[Theorem 2.1]{chi_multipush}, on which we need not rely): the inclusion functor $\cqg_c\to \cqg$
  \begin{itemize}
  \item is a right adjoint so on the one hand it preserves pullbacks and hence diagrams of the form \Cref{eq:h'h'h'};
  \item and also preserves split coequalizers, as in the proof of \Cref{le:reflok} \Cref{item:ordcoeq}.
  \end{itemize}
  This means in particular that said inclusion functor preserves codominions.
\end{remark}

\section{Asides on scalar extension and limits of coalgebras}\label{se:limtens}

In light of the proof of \Cref{pr:fieldextcodom}, one might wonder whether it could have been phrased in terms of the other description of a coalgebra codominion: as a coequalizer of the two {\it coalgebra} (rather than vector space) morphisms on the left-hand side of \Cref{eq:cxc}, with the symbol $\times_D$ denoting the pullback in $\Coalg$.

The algebra analogue, whereby the dominion of $f:A\to B$ is the equalizer of the two morphisms
\begin{equation*}
  \begin{tikzpicture}[auto,baseline=(current  bounding  box.center)]
    \path[anchor=base] 
    (0,0) node (l) {$B$}
    +(3,0) node (r) {$B\coprod_AB$}
    ;
    \draw[->] (l) to[bend left=6] node[pos=.5,auto] {$\scriptstyle $} (r);
    \draw[->] (l) to[bend right=6] node[pos=.5,auto] {$\scriptstyle $} (r);
  \end{tikzpicture}
\end{equation*}
(pushout in $\Alg$) is obvious enough:
\begin{equation*}
  \Alg_{\Bbbk}\xrightarrow{\Bbbk'\otimes_{\Bbbk}-}\Alg_{\Bbbk'}
\end{equation*}
is a left adjoint so preserves pushouts, as it does (co)equalizers. Note, though an asymmetry in the behavior of $\Bbbk'\otimes_{\Bbbk}-$ for algebras vs. coalgebras: in both cases the functor $\Bbbk'\otimes_{\Bbbk}-$ is a {\it left} adjoint, so the customary dualization does not obtain. Indeed, as observed in \Cref{se:prel}, it is enough to argue that the functor is (in both cases) cocontinuous:
\begin{itemize}
\item For coalgebras this is obvious, since the colimits in $\Coalg$ are just those in $\vc$ (\Cref{re:epico-monoalg}).
\item While for algebras one can argue directly that coequalizers and coproducts are preserved, e.g. using the convenient description of coproducts \cite[Corollary 8.1]{bg-diamond} as direct sums of tensor products.
\end{itemize}    
That nevertheless the variant of the proof of \Cref{pr:fieldextcodom} alluded to above would have gone through follows from \Cref{th:fieldextfinlim}: even though it is a left adjoint and not a right one, it does preserve the $\Coalg$-equalizer of 
\begin{equation*}
  \begin{tikzpicture}[auto,baseline=(current  bounding  box.center)]
    \path[anchor=base] 
    (0,0) node (l) {$C\times_DC$}
    +(3,0) node (r) {$C$}
    ;
    \draw[->] (l) to[bend left=6] node[pos=.5,auto] {$\scriptstyle $} (r);
    \draw[->] (l) to[bend right=6] node[pos=.5,auto] {$\scriptstyle $} (r);
  \end{tikzpicture}
\end{equation*}
for a coalgebra morphism $C\to D$.

\begin{theorem}\label{th:fieldextfinlim}
  Let $\Bbbk\le \Bbbk'$ be a field extension, and consider the left adjoint $(-)_{\Bbbk'}:=\Bbbk'\otimes_{\Bbbk}-$ between the two versions (of $\Bbbk$- and $\Bbbk'$-linear objects respectively) of any of the categories
  \begin{equation*}
    \Alg_{\bullet},\quad
    \Coalg_{\bullet},\quad
    \Bialg_{\bullet},\quad
    {}_{\square}\HAlg_{\bullet}.
  \end{equation*}
  \begin{enumerate}[(1)]

  \item\label{item:infprodalg} For $\Alg_{\bullet}$, $(-)_{\Bbbk'}$ preserves  infinite products (and hence is also right adjoint) precisely when $\Bbbk\le \Bbbk'$ is a finite extension.

    Otherwise, the functor does not preserve {\it any} infinite product of non-zero algebras (and hence is not right adjoint).

  \item\label{item:coalg} For all other categories $(-)_{\Bbbk'}$ preserves equalizers of arbitrary families and arbitrary products, and hence is a right adjoint.
    
  \end{enumerate}
\end{theorem}
\begin{proof}
  As noted repeatedly (in various formulations; e.g. \cite[diagram (9)]{porst_formal-2} and \cite[\S 2.2]{porst_univ}), the forgetful functors
  \begin{equation*}
    \begin{aligned}
      \Bialg_{\bullet},\ {}_{\square}\HAlg_{\bullet}&\xrightarrow{\quad} \Alg_{\bullet}\\
      \Bialg_{\bullet},\ {}_{\square}\HAlg_{\bullet}&\xrightarrow{\quad} \Coalg_{\bullet}\\
    \end{aligned}    
  \end{equation*}
  are right and left adjoints respectively (with the `$\bullet$' symbols understood to match: for algebras, if `c' appears on the left it does so on the right, and similarly for coalgebras and `cc'). It is thus enough to treat $\Alg_{\bullet}$ (algebras, commutative or plain) and $\Coalg_{\bullet}$ (coalgebras, cocommutative or plain).

  It is mentioned in \Cref{se:prel} that functions defined on any of the categories are right adjoint as soon as they are continuous and preserve certain filtered colimits, so in particular, for a functor that is already know to be {\it left} adjoint, continuity is equivalent to being a right adjoint. Additionally, continuity is also equivalent to preservation of products and pair equalizers \cite[Proposition 13.4]{ahs}, hence the claims about right adjointness.

  We now address the two claims.

  \begin{enumerate}[(1)]

  \item In $\Alg_{\bullet}$ the limits are just the ordinary ones form $\vc$, the forgetful functor $\Alg_{\bullet}\to \vc$ being right adjoint. Part \Cref{item:infprodalg}, then, follows immediately from the remark that the canonical map
    \begin{equation}\label{eq:k'prod}
      \left(\prod_{i\in i}V_i\right)_{\Bbbk'}
      \xrightarrow{\quad}
      \prod_{i\in I}V_{i,\Bbbk'}
    \end{equation}
    is
    \begin{itemize}
    \item an isomorphism if $\Bbbk\le \Bbbk'$ is a finite extension;
    \item and a proper embedding if the field extension is infinite and infinitely many $V_i$ are non-zero.
    \end{itemize}

  \item We will discuss the various types of limits separately. 

    
    \begin{enumerate}[(a)]
    \item\label{item:eq} {\bf (Equalizers)} In both $\Coalg$ and $\Coalg_{cc}$ these are computed identically: given a family of morphisms $f_i:C\to D$, their equalizer is the largest subcoalgebra of $C$ contained in the vector-space equalizer
      \begin{equation}\label{eq:veceq}
        \mathrm{Eq}_{\vc}(f_i):=\{c\in C\ |\ f_i(c)\in D\text{ are all equal}\}
      \end{equation}
      (see for instance \cite[first paragraph of the proof of Theorem 1.1]{agore_lim} for pairs of morphisms; the argument works generally). That $\Bbbk'\otimes_{\Bbbk}-$ preserves vector-space equalizers \Cref{eq:veceq} is straightforward linear algebra, so the conclusion follows from \Cref{le:lgsubcoalg} below.
      
    \item\label{item:prodcoalg} {\bf (Products in $\Coalg$)} Let $C_i$, $i\in I$ be a family of $\Bbbk$-coalgebras, and consider the canonical morphism
      \begin{equation}\label{eq:canprodmap}
        \left(\prod C_i\right)_{\Bbbk'}
        \xrightarrow{\quad\cat{can}\quad}
        \prod C_{i,\Bbbk'},
      \end{equation}
      where the products are taken in the categories of $\Bbbk$- and $\Bbbk'$-coalgebras respectively.

      Consider a finite-dimensional $\Bbbk$-vector space $V$. By \Cref{le:comodstruct}, the morphism \Cref{eq:canprodmap} induces a bijection between the comodule structures on $V_{\Bbbk'}$ over the two sides of that morphism. The same result also implies that for finite-dimensional $\Bbbk$-spaces $V$ and $W$ equipped with such comodule structures we have
      \begin{equation*}
        \cM^{C_{\Bbbk'}}(V_{\Bbbk'},W_{\Bbbk'})
        \cong
        \cM^C(V,W)_{\Bbbk'},\quad
        C:=\prod C_i.
      \end{equation*}
      Given that every finite-dimensional $\Bbbk'$-vector space is of the form $V_{\Bbbk'}$ for some $V\in \vc_{\Bbbk,f}$, the {\it Tannakian reconstruction} of the coalgebras \Cref{eq:canprodmap} from their respective finite-dimensional comodules \cite[Theorem 2.1.12 and Lemma 2.2.1]{schau_tann} now makes it clear that \Cref{eq:canprodmap} is an isomorphism.

    \item {\bf (Products in $\Coalg_{cc}$)} Consider an arbitrary family of cocommutative coalgebras $C_i\in \Coalg_{cc}$, $i\in I$. For any {\it finite} subfamily thereof the product in $\Coalg$ is nothing but the tensor product \cite[Theorem 6.4.5]{swe}, so in general we have
      \begin{equation}\label{eq:cocomprod}
        \prod^{\Coalg_{cc}}C_i\cong \varprojlim^{\Coalg}_{F} \bigotimes_{i\in F}C_i,
      \end{equation}
      where
      \begin{itemize}
      \item $F$ ranges over the finite subsets of the full index set $I$, ordered by inclusion
      \item For $F\subseteq F'$ the connecting morphism
        \begin{equation*}
          \bigotimes_{i\in F'}C_i
          \xrightarrow{\quad}
          \bigotimes_{i\in F}C_i
        \end{equation*}
        acts as the identity for $i\in F$ and as the counit for $i\in F'\setminus F$.
      \end{itemize}
      As we already know from points \Cref{item:eq} and \Cref{item:prodcoalg} above that $(-)_{\Bbbk'}$ preserves arbitrary limits in $\Coalg$, and since it of course preserves tensor products, it will also preserve limits of the form \Cref{eq:cocomprod}.

    \end{enumerate}
    

  \end{enumerate}

  This concludes the proof.
\end{proof}

\begin{lemma}\label{le:lgsubcoalg}
  For a linear subspace $V\le C$, denote by $C_V\le V$ the largest subcoalgebra of $C$ contained in $V$.
  
  For any field extension $\Bbbk\le \Bbbk'$ and linear subspace $V\le C$ of a $\Bbbk$-coalgebra the canonical inclusion
  \begin{equation*}
    \Bbbk'\otimes_{\Bbbk} C_V\to (\Bbbk'\otimes_{\Bbbk}C)_{\Bbbk'\otimes_{\Bbbk}V}
  \end{equation*}
  is an equality. 
\end{lemma}
\begin{proof}
  In words, scalar extensions along field inclusions preserve the largest-subcoalgebra-contained-in-a-subspace construction.
  
  To see this, observe that
  \begin{equation*}
    C_V = \{c\in C\ |\ \Delta_n c\in V^{\otimes(n+1)}\le C^{\otimes(n+1)}\},
  \end{equation*}
  where
  \begin{equation*}
    C\ni c\xmapsto{\quad\Delta_n\quad}c_1\otimes\cdots\otimes c_{n+1}\in C^{\otimes(n+1)}
  \end{equation*}
  is the iterated comultiplication (so that $\Delta_1=\Delta_C$ and $\Delta_0=\id_C$). We thus have 
  \begin{equation*}
    C_V = \bigcap_{n\ge 0}\Delta_n^{-1}V^{\otimes(n+1)},
  \end{equation*}  
  and we conclude by noting that $\Bbbk'\otimes_{\Bbbk}-$ preserves intersections of arbitrary families of subspaces (e.g. by choosing a basis for $\Bbbk'$ over $\Bbbk$, etc.).
\end{proof}

\begin{lemma}\label{le:comodstruct}
  Let $\Bbbk\le \Bbbk'$ be a field extension, $C$ a $\Bbbk$-coalgebra, $V$ a $\Bbbk$-vector space, and denote by $(-)_{\Bbbk'}$ the functor $\Bbbk'\otimes_{\Bbbk}-$. 

  A right $C_{\Bbbk'}$-comodule structure on $V_{\Bbbk'}$ arises by scalar extension from a unique $C$-comodule structure on $V$.
\end{lemma}
\begin{proof}
  Per the adjunction
  \begin{equation*}
    \begin{tikzpicture}[auto,baseline=(current  bounding  box.center)]
      \path[anchor=base] 
      (0,0) node (l) {$\vc_{\Bbbk}$}
      +(4,0) node (r) {$\vc_{\Bbbk'}$,}
      ;
      \draw[->] (l) to[bend left=10] node[pos=.5,auto] {$\scriptstyle (-)_{\Bbbk'}$} (r);
      \draw[->] (r) to[bend left=10] node[pos=.5,auto] {$\scriptstyle \cat{scalar restriction}$} (l);
    \end{tikzpicture}
  \end{equation*}
  a $\Bbbk'$-linear map $V_{\Bbbk'}\to V_{\Bbbk'}\otimes_{\Bbbk'}C_{\Bbbk'}$ is (the same thing as) a $\Bbbk$-linear map $V\to V\otimes C\otimes\Bbbk'$ (unadorned tensor products being over $\Bbbk$). Choosing a basis $e_i$, $i\in I$ for $\Bbbk'/\Bbbk$ with $1=e_{i_0}$, such a map is in turn uniquely determined by its components
  \begin{equation*}
    V\ni v\xmapsto{\quad}v^{(i)}_0\otimes v^{(i)}_1\otimes e_i\in V\otimes C\otimes K. 
  \end{equation*}
  The condition that the original map be a $C_{\Bbbk'}$-comodule structure then amounts to
  \begin{itemize}
  \item the map
    \begin{equation*}
      V\ni v\xmapsto{\quad} v^{(i_0)}_0\otimes v^{(i_0)}_1\in V\otimes C
    \end{equation*}
    being a right $C$-comodule structure;
  \item while
    \begin{equation}\label{eq:awayi0}
      V\ni v\xmapsto{\quad} v^{(i)}_0\otimes v^{(i)}_1\in V\otimes C,\ i\ne i_0
    \end{equation}
    are coassociative and satisfy
    \begin{equation}\label{eq:anticounit}
      v^{(i)}_0\varepsilon\left(v^{(i)}_1\right)=0\in V,\ \forall v\in V
    \end{equation}
    (by contrast to the usual counitality condition, that would equate the left-hand side with $v$). 
  \end{itemize}
  Coassociativity together with \Cref{eq:anticounit} are easily seen to imply the vanishing of all \Cref{eq:awayi0}, proving the claim: the $C_{\Bbbk'}$-comodule structures on $V_{\Bbbk'}$ are precisely those of the form
  \begin{equation*}
    V_{\Bbbk'}\ge V\ni v\xmapsto{\quad}v_0\otimes v_1\otimes 1\in V\otimes C\otimes \Bbbk'\cong V_{\Bbbk'}\otimes_{\Bbbk'}C_{\Bbbk'}
  \end{equation*}
  for a $C$-comodule structure $v\mapsto v_0\otimes v_1$ on $V$. 
\end{proof}



\addcontentsline{toc}{section}{References}

\begin{thebibliography}{10}

\bibitem{ahs}
Ji\v{r}\'{\i} Ad\'{a}mek, Horst Herrlich, and George~E. Strecker.
\newblock {\em Abstract and concrete categories}.
\newblock Pure and Applied Mathematics (New York). John Wiley \& Sons, Inc.,
  New York, 1990.
\newblock The joy of cats, A Wiley-Interscience Publication.

\bibitem{ar}
Ji\v{r}\'{\i} Ad\'{a}mek and Ji\v{r}\'{\i} Rosick\'{y}.
\newblock {\em Locally presentable and accessible categories}, volume 189 of
  {\em London Mathematical Society Lecture Note Series}.
\newblock Cambridge University Press, Cambridge, 1994.

\bibitem{agore_mono}
A.~L. Agore.
\newblock Monomorphisms of coalgebras.
\newblock {\em Colloq. Math.}, 120(1):149--155, 2010.

\bibitem{agore_lim}
A.~L. Agore.
\newblock Limits of coalgebras, bialgebras and {H}opf algebras.
\newblock {\em Proc. Amer. Math. Soc.}, 139(3):855--863, 2011.

\bibitem{bw}
Michael Barr and Charles Wells.
\newblock {\em Toposes, triples and theories}, volume 278 of {\em Grundlehren
  der mathematischen Wissenschaften [Fundamental Principles of Mathematical
  Sciences]}.
\newblock Springer-Verlag, New York, 1985.

\bibitem{bg_epi-lie}
George~M. Bergman.
\newblock Epimorphisms of lie algebras, 1973.
\newblock available at
  \url{https://math.berkeley.edu/~gbergman/papers/unpub/LieEpi.pdf} (accessed
  2023-02-07).

\bibitem{bg-diamond}
George~M. Bergman.
\newblock The diamond lemma for ring theory.
\newblock {\em Adv. in Math.}, 29(2):178--218, 1978.

\bibitem{bhm_ext}
A.~Bia{\l}ynicki-Birula, G.~Hochschild, and G.~D. Mostow.
\newblock Extensions of representations of algebraic linear groups.
\newblock {\em Amer. J. Math.}, 85:131--144, 1963.

\bibitem{bichon_expect}
Julien Bichon.
\newblock Faithful flatness of hopf algebras over coideal subalgebras with a
  conditional expectation, 2023.
\newblock http://arxiv.org/abs/2301.05480v1.

\bibitem{bb1}
Fr\'{e}d\'{e}ric Bien and Armand Borel.
\newblock Sous-groupes \'{e}pimorphiques de groupes lin\'{e}aires
  alg\'{e}briques. {I}.
\newblock {\em C. R. Acad. Sci. Paris S\'{e}r. I Math.}, 315(6):649--653, 1992.

\bibitem{bb2}
Fr\'{e}d\'{e}ric Bien and Armand Borel.
\newblock Sous-groupes \'{e}pimorphiques de groupes lin\'{e}aires
  alg\'{e}briques. {II}.
\newblock {\em C. R. Acad. Sci. Paris S\'{e}r. I Math.}, 315(13):1341--1346,
  1992.

\bibitem{brl}
Armand Borel.
\newblock {\em Linear algebraic groups}, volume 126 of {\em Graduate Texts in
  Mathematics}.
\newblock Springer-Verlag, New York, second edition, 1991.

\bibitem{brion_epi}
Michel Brion.
\newblock Epimorphic subgroups of algebraic groups.
\newblock {\em Math. Res. Lett.}, 24(6):1649--1665, 2017.

\bibitem{bd_lie}
Theodor Br\"{o}cker and Tammo tom Dieck.
\newblock {\em Representations of compact {L}ie groups}, volume~98 of {\em
  Graduate Texts in Mathematics}.
\newblock Springer-Verlag, New York, 1985.

\bibitem{brz-wis}
Tomasz Brzezinski and Robert Wisbauer.
\newblock {\em Corings and comodules}, volume 309 of {\em London Mathematical
  Society Lecture Note Series}.
\newblock Cambridge University Press, Cambridge, 2003.

\bibitem{chi_cos-xv1}
Alexandru Chirvasitu.
\newblock Cosemisimple hopf algebras are faithfully flat over hopf subalgebras,
  2011.
\newblock http://arxiv.org/abs/1110.6701v1.

\bibitem{chi_cos}
Alexandru Chirvasitu.
\newblock Cosemisimple {H}opf algebras are faithfully flat over {H}opf
  subalgebras.
\newblock {\em Algebra Number Theory}, 8(5):1179--1199, 2014.

\bibitem{chi_cat-cqg}
Alexandru Chirvasitu.
\newblock Categorical aspects of compact quantum groups.
\newblock {\em Appl. Categ. Structures}, 23(3):381--413, 2015.

\bibitem{chi_multipush}
Alexandru Chirvasitu.
\newblock Non-degeneracy results for (multi-)pushouts of compact groups, 2023.
\newblock http://arxiv.org/abs/2301.09847v2.

\bibitem{chi_epi}
Alexandru Chirv\u{a}situ.
\newblock On epimorphisms and monomorphisms of {H}opf algebras.
\newblock {\em J. Algebra}, 323(5):1593--1606, 2010.

\bibitem{cip_ff}
Michele Cipolla.
\newblock Discesa fedelmente piatta dei moduli.
\newblock {\em Rend. Circ. Mat. Palermo (2)}, 25(1-2):43--46, 1976.

\bibitem{del_tann}
P.~Deligne.
\newblock Cat\'{e}gories tannakiennes.
\newblock In {\em The {G}rothendieck {F}estschrift, {V}ol. {II}}, volume~87 of
  {\em Progr. Math.}, pages 111--195. Birkh\"{a}user Boston, Boston, MA, 1990.

\bibitem{dk_cqg}
Mathijs~S. Dijkhuizen and Tom~H. Koornwinder.
\newblock C{QG} algebras: a direct algebraic approach to compact quantum
  groups.
\newblock {\em Lett. Math. Phys.}, 32(4):315--330, 1994.

\bibitem{dnr}
Sorin D\u{a}sc\u{a}lescu, Constantin N\u{a}st\u{a}sescu, and \c{S}erban Raianu.
\newblock {\em Hopf algebras}, volume 235 of {\em Monographs and Textbooks in
  Pure and Applied Mathematics}.
\newblock Marcel Dekker, Inc., New York, 2001.
\newblock An introduction.

\bibitem{egno_tensor}
Pavel Etingof, Shlomo Gelaki, Dmitri Nikshych, and Victor Ostrik.
\newblock {\em Tensor categories}, volume 205 of {\em Mathematical Surveys and
  Monographs}.
\newblock American Mathematical Society, Providence, RI, 2015.

\bibitem{ev_coh}
Leonard Evens.
\newblock {\em The cohomology of groups}.
\newblock Oxford Mathematical Monographs. The Clarendon Press, Oxford
  University Press, New York, 1991.
\newblock Oxford Science Publications.

\bibitem{ghklms_latt}
G.~Gierz, K.~H. Hofmann, K.~Keimel, J.~D. Lawson, M.~Mislove, and D.~S. Scott.
\newblock {\em Continuous lattices and domains}, volume~93 of {\em Encyclopedia
  of Mathematics and its Applications}.
\newblock Cambridge University Press, Cambridge, 2003.

\bibitem{gross_homog}
Frank~D. Grosshans.
\newblock {\em Algebraic homogeneous spaces and invariant theory}, volume 1673
  of {\em Lecture Notes in Mathematics}.
\newblock Springer-Verlag, Berlin, 1997.

\bibitem{Hart}
R.~Hartshorne.
\newblock {\em Algebraic geometry}.
\newblock Springer-Verlag, New York-Heidelberg, 1977.
\newblock Graduate Texts in Mathematics, No. 52.

\bibitem{hn-epi}
K.~H. Hofmann and K.-H. Neeb.
\newblock Epimorphisms of {$C^*$}-algebras are surjective.
\newblock {\em Arch. Math. (Basel)}, 65(2):134--137, 1995.

\bibitem{isb_epidom-2}
J.~M. Howie and J.~R. Isbell.
\newblock Epimorphisms and dominions. {II}.
\newblock {\em J. Algebra}, 6:7--21, 1967.

\bibitem{hmph_intro}
James~E. Humphreys.
\newblock {\em Introduction to {L}ie algebras and representation theory}.
\newblock Graduate Texts in Mathematics, Vol. 9. Springer-Verlag, New
  York-Berlin, 1972.

\bibitem{isb_epidom-5}
J.~R. Isbell.
\newblock Epimorphisms and dominions, {V}.
\newblock {\em Algebra Universalis}, 3:318--320, 1973.

\bibitem{isb_epidom-1}
John~R. Isbell.
\newblock Epimorphisms and dominions.
\newblock In {\em Proc. {C}onf. {C}ategorical {A}lgebra ({L}a {J}olla,
  {C}alif., 1965)}, pages 232--246. Springer, New York, 1966.

\bibitem{isb_epidom-3}
John~R. Isbell.
\newblock Epimorphisms and dominions. {III}.
\newblock {\em Amer. J. Math.}, 90:1025--1030, 1968.

\bibitem{isb_epidom-4}
John~R. Isbell.
\newblock Epimorphisms and dominions. {IV}.
\newblock {\em J. London Math. Soc. (2)}, 1:265--273, 1969.

\bibitem{jc}
Nathan Jacobson.
\newblock {\em Lie algebras}.
\newblock Dover Publications, Inc., New York, 1979.
\newblock Republication of the 1962 original.

\bibitem{krau_ausl}
Henning Krause.
\newblock Deriving {A}uslander's formula.
\newblock {\em Doc. Math.}, 20:669--688, 2015.

\bibitem{lam_mod-rng}
T.~Y. Lam.
\newblock {\em Lectures on modules and rings}, volume 189 of {\em Graduate
  Texts in Mathematics}.
\newblock Springer-Verlag, New York, 1999.

\bibitem{mcl}
Saunders Mac~Lane.
\newblock {\em Categories for the working mathematician}, volume~5 of {\em
  Graduate Texts in Mathematics}.
\newblock Springer-Verlag, New York, second edition, 1998.

\bibitem{mesabl_desc}
Bachuki Mesablishvili.
\newblock Monads of effective descent type and comonadicity.
\newblock {\em Theory Appl. Categ.}, 16:No. 1, 1--45, 2006.

\bibitem{montg_hopf}
Susan Montgomery.
\newblock {\em Hopf algebras and their actions on rings}, volume~82 of {\em
  CBMS Regional Conference Series in Mathematics}.
\newblock Published for the Conference Board of the Mathematical Sciences,
  Washington, DC, 1993.

\bibitem{mon}
Susan Montgomery.
\newblock {\em Hopf algebras and their actions on rings}, volume~82 of {\em
  CBMS Regional Conference Series in Mathematics}.
\newblock Published for the Conference Board of the Mathematical Sciences,
  Washington, DC; by the American Mathematical Society, Providence, RI, 1993.

\bibitem{fkm}
D.~Mumford, J.~Fogarty, and F.~Kirwan.
\newblock {\em Geometric invariant theory}, volume~34 of {\em Ergebnisse der
  Mathematik und ihrer Grenzgebiete (2) [Results in Mathematics and Related
  Areas (2)]}.
\newblock Springer-Verlag, Berlin, third edition, 1994.

\bibitem{nt_torsion}
C.~N\u{a}st\u{a}sescu and B.~Torrecillas.
\newblock Torsion theories for coalgebras.
\newblock {\em J. Pure Appl. Algebra}, 97(2):203--220, 1994.

\bibitem{nuss_ff}
Philippe Nuss.
\newblock Noncommutative descent and non-abelian cohomology.
\newblock {\em $K$-Theory}, 12(1):23--74, 1997.

\bibitem{pet_epi}
A.~V. Petukhov.
\newblock A geometric description of epimorphic subgroups.
\newblock {\em Uspekhi Mat. Nauk}, 65(5(395)):193--194, 2010.

\bibitem{pogunt_epi-cpct}
Detlev Poguntke.
\newblock Epimorphisms of compact groups are onto.
\newblock {\em Proc. Amer. Math. Soc.}, 26:503--504, 1970.

\bibitem{pop}
N.~Popescu.
\newblock {\em Abelian categories with applications to rings and modules}.
\newblock Academic Press, London-New York, 1973.
\newblock London Mathematical Society Monographs, No. 3.

\bibitem{porst_bimon}
Hans-E. Porst.
\newblock On categories of monoids, comonoids, and bimonoids.
\newblock {\em Quaest. Math.}, 31(2):127--139, 2008.

\bibitem{porst_univ}
Hans-E. Porst.
\newblock Universal constructions for {H}opf algebras.
\newblock {\em J. Pure Appl. Algebra}, 212(11):2547--2554, 2008.

\bibitem{porst_formal-2}
Hans-E. Porst.
\newblock The formal theory of {H}opf algebras {P}art {II}: {T}he case of
  {H}opf algebras.
\newblock {\em Quaest. Math.}, 38(5):683--708, 2015.

\bibitem{rad}
David~E. Radford.
\newblock {\em Hopf algebras}, volume~49 of {\em Series on Knots and
  Everything}.
\newblock World Scientific Publishing Co. Pte. Ltd., Hackensack, NJ, 2012.

\bibitem{reid-epi}
G.~A. Reid.
\newblock Epimorphisms and surjectivity.
\newblock {\em Invent. Math.}, 9:295--307, 1969/70.

\bibitem{schau_tann}
Peter Schauenburg.
\newblock {\em Tannaka duality for arbitrary {H}opf algebras}, volume~66 of
  {\em Algebra Berichte [Algebra Reports]}.
\newblock Verlag Reinhard Fischer, Munich, 1992.

\bibitem{schau_counter}
Peter Schauenburg.
\newblock Faithful flatness over {H}opf subalgebras: counterexamples.
\newblock In {\em Interactions between ring theory and representations of
  algebras ({M}urcia)}, volume 210 of {\em Lecture Notes in Pure and Appl.
  Math.}, pages 331--344. Dekker, New York, 2000.

\bibitem{stens_quot}
Bo~Stenstr{\"o}m.
\newblock {\em Rings of quotients}.
\newblock Springer-Verlag, New York, 1975.
\newblock Die Grundlehren der Mathematischen Wissenschaften, Band 217, An
  introduction to methods of ring theory.

\bibitem{swe}
Moss~E. Sweedler.
\newblock {\em Hopf algebras}.
\newblock Mathematics Lecture Note Series. W. A. Benjamin, Inc., New York,
  1969.

\bibitem{tak_rel}
Mitsuhiro Takeuchi.
\newblock Relative {H}opf modules---equivalences and freeness criteria.
\newblock {\em J. Algebra}, 60(2):452--471, 1979.

\bibitem{bt_rat}
Nguy\^{e}\~{n}~Qu\^{o}\'{c} Th\v{a}\'{n}g and Dao~Phuong B\v{a}\'{c}.
\newblock Some rationality properties of observable groups and related
  questions.
\newblock {\em Illinois J. Math.}, 49(2):431--444, 2005.

\bibitem{wang_fp}
Shuzhou Wang.
\newblock Free products of compact quantum groups.
\newblock {\em Comm. Math. Phys.}, 167(3):671--692, 1995.

\bibitem{water_bk}
William~C. Waterhouse.
\newblock {\em Introduction to affine group schemes}, volume~66 of {\em
  Graduate Texts in Mathematics}.
\newblock Springer-Verlag, New York-Berlin, 1979.

\end{thebibliography}

\Addresses

\end{document}